\documentclass[11pt, a4paper, english]{smfart}

\usepackage[utf8]{inputenc}

\usepackage{amssymb,amsfonts,amsmath,amsthm,amsopn,amstext,amscd,latexsym,xy,stmaryrd,mathtools,xcolor,enumitem}
\usepackage[hidelinks]{hyperref}

\usepackage{tikz-cd}
\tikzcdset{scale cd/.style={every label/.append style={scale=#1},
    cells={nodes={scale=#1}}}}
\usepackage{tikz}
\usetikzlibrary{patterns,shapes,arrows}

\input xy
\xyoption{all}

\swapnumbers

\newtheoremstyle{numpar}
{\topsep}			
{\topsep}			
{}			
{}			
{\bfseries}			
{ }			
{5pt plus 1pt minus 1pt}			
{}			

\theoremstyle{numpar}
\newtheorem{para}[subsection]{\unskip}

\newtheoremstyle{mythm}
  {\topsep}   
  {\topsep}   
  {\itshape}  
  {0pt}       
  {\bfseries} 
  {.}         
  {5pt plus 1pt minus 1pt} 
  {}          

\theoremstyle{mythm}

\newtheorem{theorem}[subsection]{Theorem}
\newtheorem*{introthm*}{Theorem}
\newtheorem*{introcor*}{Corollary}
\newtheorem{lemma}[subsection]{Lemma}
\newtheorem{proposition}[subsection]{Proposition}

\newtheorem{corollary}[subsection]{Corollary}

\newtheoremstyle{myrmk}
  {\topsep}   
  {\topsep}   
  {}  
  {0pt}       
  {\bfseries} 
  {.}         
  {5pt plus 1pt minus 1pt} 
  {}          
  
\theoremstyle{myrmk}
\newtheorem{definition}[subsection]{Definition}

\theoremstyle{myrmk}
\newtheorem{remark}[subsection]{Remark}

\numberwithin{equation}{section}

\DeclareMathOperator{\GL}{GL}

\DeclareMathOperator{\depth}{depth}

\DeclareMathOperator{\Def}{Def}

\DeclareMathOperator{\ord}{ord}

\DeclareMathOperator{\Frob}{Frob}
\DeclareMathOperator{\Sym}{Sym}
\DeclareMathOperator{\Hom}{Hom}

\DeclareMathOperator{\Gal}{Gal}
\DeclareMathOperator{\Aut}{Aut}

\DeclareMathOperator{\im}{im}

\DeclareMathOperator{\ad}{ad}

\DeclareMathOperator{\End}{End}
\DeclareMathOperator{\tr}{tr}

\DeclareMathOperator{\CNL}{CNL}

\DeclareMathOperator{\Res}{Res}

\DeclareMathOperator{\Art}{Art}

\DeclareMathOperator{\Iw}{Iw}
\DeclareMathOperator{\rec}{rec}

\newcommand{\cO}{{\mathcal O}}

\newcommand{\cV}{{\mathcal V}}

\newcommand{\bbA}{{\mathbb A}}

\newcommand{\bbC}{{\mathbb C}}

\newcommand{\bbQ}{{\mathbb Q}}

\allowdisplaybreaks

\title[Adjoint Bloch--Kato Selmer groups]{Adjoint Bloch--Kato Selmer groups of regular algebraic automorphic Galois representations}
\author{L. A'Campo \and B. Hevesi \and J. A. Thorne \and D. Whitmore }

\begin{document}
\begin{abstract}
    We prove the vanishing of the adjoint Bloch--Kato Selmer group of the Galois representations associated to regular algebraic automorphic representations of general linear groups over CM fields. A key novelty of our work is that we impose conditions only on the $p$-adic Galois representations, and not on their associated residual representations modulo $p$. 
\end{abstract}

\maketitle
\tableofcontents
\section{Introduction}

Let $F$ be a number field, let $p$ be a prime, and let 
\[ \rho : \Gal(\overline{F} / F) \to \GL_n(\overline{\mathbb{Q}}_p) \]
be a continuous, irreducible representation that is geometric, in the sense that it is unramified at all but finitely many places, and its restriction to any decomposition group at a $p$-adic place of $F$ is de Rham. (For example, this condition is satisfied if $\rho$ is an irreducible subquotient of the $p$-adic \'etale cohomology groups of a smooth, projective algebraic variety over $F$.) 

Let $\ad \rho$ denote the adjoint representation of $\rho$ (of $\Gal(\overline{F} / F)$ on the space of $n \times n$ matrices, by conjugation). General conjectures of Fontaine--Mazur, Beilinson, and Bloch--Kato \cite{Fon95, Bei84, Blo90} predict the vanishing of the Bloch--Kato Selmer group
\[ H^1_f(F, \ad \rho) = \ker( H^1(F, \ad \rho) \to \prod_{v \nmid p} \frac{H^1(F_v, \ad \rho)}{H^1_{ur}(F_v, \ad \rho)} \times \prod_{v | p} H^1(F_v, \ad \rho \otimes_{\mathbb{Q}_p} B_{cris}) ). \]
Unconditional results of this type are of interest both as instances of an important conjecture, and because of their applications to Galois deformation theory and the Langlands programme (see e.g.\ \cite{New21, New26} for applications to symmetric power functoriality). In this paper, we establish new cases of Bloch--Kato vanishing for the Galois representations attached to cuspidal, regular algebraic automorphic representations of $\GL_n(\mathbb{A}_F)$, where $F$ is a CM number field:
\begin{introthm*}\label{introthm_main_result}
    Let $F$ be a CM number field, let $\pi$ be a cuspidal, regular algebraic automorphic representation of $\GL_n(\bbA_F)$, and let $\iota : \overline{\bbQ}_p \to \bbC$ be an isomorphism. Suppose that the following conditions are satisfied:
    \begin{enumerate}
        \item The image $r_{\pi,\iota}(G_{F(\zeta_{p^\infty})})$ is enormous, in the sense of \cite[Definition 2.23]{New23}.
        \item For all finite places $v$ of $F$, the Weil--Deligne representation $\mathrm{WD}(r_{\pi, \iota}|_{G_{F_v}})$ is pure, in the sense of \cite[\S 1]{Tay07}.
    \end{enumerate}
    Then $H^1_f(F, \ad r_{\pi, \iota}) = 0$.
\end{introthm*}
We record some corollaries with easily verified hypotheses.
\begin{introcor*}
    Let $F$ be a CM number field, and let $\pi$ be a cuspidal, regular algebraic automorphic representation of $\GL_n(\bbA_F)$ that is polarizable, in the sense of \cite{Bar14}. Let $\iota : \overline{\bbQ}_p \to \bbC$ be an isomorphism. Suppose there exists a finite place $v$ of $F$ such that $\pi_v$ is a twist of the Steinberg representation. Then $H^1_f(F, \ad r_{\pi, \iota}) = 0$.
\end{introcor*}
(For polarizable representations are known to be pure, and the `enormous image' condition follows from \cite[Example 2.30]{New23}. Note that, although $\pi$ is assumed to be polarizable, the above result refers to the whole adjoint representation over $F$, rather than just its descent to the maximal totally real subfield $F^+$ of $F$, as considered in \cite[Introduction]{New23}.) 
\begin{introcor*}
    Let $F$ be a CM number field, and let $E$ be an elliptic curve over $F$ with $\End(E_{\overline{F}}) = \mathbb{Z}$. Let $p$ be a prime, and let $\rho_{E, p} : \Gal(\overline{F} / F) \to \GL_2(\mathbb{Q}_p)$ be the representation afforded by the $p$-adic Tate module of $E$. Then for any $n \geq 1$, $H^1_f(F, \ad \Sym^n \rho_{E, p}) = 0$.
\end{introcor*}
(For $\rho_{E, p}$ is known to be pure, and this property is preserved under tensor products; its symmetric powers are potentially automorphic, by \cite{10author}; and the enormous image condition follows from Serre's open image theorem.) 

The above theorem is stated below as Corollary \ref{cor_selmer_vanishing_under_purity}. Our most general result is Theorem \ref{thm_general_vanishing_result}, which however holds under slightly more technical conditions; we invite the reader to look ahead before continuing. 

We now compare the above theorem to other results in the literature. Existing results of a similar nature can be separated into two classes, the `Taylor--Wiles' or `defect 0' case, and the `Calegari--Geraghty' or `positive defect' case. Since the Bloch--Kato Selmer group $H^1_f$ can be identified with the Zariski tangent space of a suitable Galois deformation ring at the point corresponding to $\rho$, its vanishing often follows from suitable $R = \mathbb{T}$ (or $R[1/p] = \mathbb{T}[1/p]$) theorems, showing that a Galois deformation ring may be identified with a Hecke algebra whose $\mathbb{Q}_p$-fibre is known to be \'etale. Similar ideas can sometimes be used to prove the vanishing of $H^1_f$ even in cases where it is not possible to prove an $R = \mathbb{T}$ theorem for the associated residual representation $\overline{\rho}$. 

This was the approach taken in \cite{kisin2003}, which features the requirement that $\overline{\rho}|_{G_{F(\zeta_p)}}$ is non-trivial. The key difficulty in removing this condition is that, without it, one loses control of the dimension of the spaces of automorphic forms that appear when introducing auxiliary Taylor--Wiles places. Pan \cite{Pan22} showed how to use ultrafilters to circumvent this difficulty, and the works \cite{New23, Tho22} pursued Pan's approach to establish the analogue of the above theorem for the `self-dual' part of the Bloch--Kato Selmer group in the polarizable case with essentially no condition other than irreducibility (noting that purity is known in this case). 

To establish analogous results in the `positive defect' setting, one wants to mesh the ideas of the previous paragraph with the approach given by Calegari--Geraghty \cite{Cal18}, where one of the key difficulties is that one has to deal with cohomology groups that have a significant amount of torsion; and ultimately, to deal with complexes, defined only up to isomorphism in the derived category, instead of simply cohomology groups. There are technical hurdles both in terms of foundations, and in terms of the techniques to be applied. 

To get started, it is necessary to know some form of local-global compatibility. Results of this type have been established in the case we consider in e.g. \cite{10author, CN25, Hev24}, but only under the `decomposed generic' condition introduced in \cite{CS17} in order to simplify the analysis of the decomposition of cohomology with respect to the Hodge--Tate period map, and under an irreducibility condition on $\overline{\rho}$ that effectively annihilates the contribution of the boundary to the cohomology of the $\GL_{n, F}$-locally symmetric spaces. The theorems we prove here are made accessible by our companion papers, which prove unconditional local-global compatibility results in the level of generality required (see \cite{Hev25} for the case $\ell \neq p$ and \cite{Aca25} for the case $\ell = p$). 

Aside from these difficulties, one has to contend with the fact that the presence of complexes (as opposed to modules) in the patching argument means that at present we can prove `$R = \mathbb{T}$'-type results only in a rather weak form, for example up to a nilpotent ideal; and that without assuming that the residual representation satisfies the usual Taylor--Wiles condition, it is far from obvious how to get a handle on the properties of the patched complexes (for example, to know that they are ultimately finitely generated modules over a suitable ring of diamond operators).

Proving $R = \mathbb{T}$ up to a nilpotent error term is insignificant for automorphy results, which refer only to geometric points of $R$, but clearly problematic for an analysis of Zariski tangent spaces. A technique for dealing with this possible nilpotent kernel was introduced in \cite{AC23}, making use of the smoothness of the $\mathbb{Q}_p$-fibre of the local deformation rings at generic points. (This genericity condition is the source of the genericity hypothesis on $\rho$ in Theorem \ref{thm_general_vanishing_result} below.) We paraphrase the essential part of this argument as Lemma \ref{lem_patching_lemma}. 

The main technical difficulty facing us here is therefore to explain how to carry out the ultra-patching process for complexes, whose cohomology groups may have unbounded rank, while also being able to apply the ideas of \cite{AC23}. The argument we give along these lines, in \S \ref{sec_patching}, is the main novelty of this paper.  The details of this argument are quite intricate, so we give an informal sketch here.

The simple Lemma \ref{lem:filtration_generator_bound} allows us to control the finiteness properties of spaces $\delta_{Q_N}^{r} H_{Q_N}$, where $H_{Q_N}$ is a space of modular forms (or cohomology classes) at auxiliary Taylor--Wiles level, and $\delta_{Q_N}$ is a `discriminant' Hecke operator of the type introduced in \cite[\S3]{New23}, measuring the difference between the modular forms at spherical and Iwahori level at the Taylor--Wiles places. Localising the modular forms at the characteristic 0 prime $\mathfrak{q}$ of the Hecke algebra corresponding to our automorphic representation would kill the difference between $H$ and $H_{Q_N}$, which would not be useful. Instead, we perform the patching argument in two steps, first replacing $H_{Q_N}$, free over a ring like $\mathbb{Z}_p[x] / (x^{p^r} - 1)$, with a patched object $\tilde{H}_r$, free over a ring like $\mathbb{Z}_p[x] / (x^{p^r})$. This has the effect that the localisation 
\[ \delta^r \widetilde{H}_{r, \mathfrak{q}} = \widetilde{H}_{r, \mathfrak{q}} \]
is now free over $\mathbb{Q}_p[x] / (x^{p^r})$, and represents a non-trivial enlargement of the initial space $H_{\mathfrak{q}}$ of characteristic 0 modular forms. (Note in particular that $\delta$ is non-zero modulo $\mathfrak{q}$, so localisation allows us to remove it from the equation.) We can then pass to the limit to recover the usual patching setup of a space of patched modular forms, free over a regular local ring like $\mathbb{Q}_p\llbracket x \rrbracket$, at which point the argument starts to follow more familiar lines. 

Finally, we mention that there are other `positive defect' settings of interest that we do not address in this paper, such as that of the Galois representations associated to modular abelian surfaces. Unconditional vanishing results for such representations are proved in \cite{Cal20, Box25}, although it remains an interesting open question to establish results with conditions only on $\rho$ (and not on $\overline{\rho}$).

\subsection*{Acknowledgements}

We thank Toby Gee for useful comments on an earlier version of this manuscript.

B.H.'s work was supported by the Engineering and Physical Sciences Research Council grant number EP/Y030648/1 and by the European Union’s Horizon 2020 research and innovation programme under the Marie Skłodowska-Curie grant agreement No 101034255.

J.T.'s work was funded by the European Union (ERC CoG-101169866). Views and opinions expressed are however those of the author(s) only and do not necessarily reflect those of the European Union or the European Research Council. Neither the European Union nor the granting authority can be held responsible for them.

\subsection*{Notation}

If $F$ is a field of characteristic zero, we generally fix an algebraic closure $\overline{F} / F$ and write $G_F$ for the absolute Galois group of $F$ with respect to this choice. If $F$ is a number field, then we will also fix embeddings $\overline{F} \to \overline{F}_v$ extending the map $F\to F_v$ for each place $v$ of $F$; this choice determines a homomorphism $G_{F_v} \to G_F$. When $v$ is a finite place, we will write $\cO_{F_v} \subset F_v$ for the valuation ring, $\varpi_v \in \cO_{F_v}$ for a fixed choice of uniformizer, $\Frob_v \in G_{F_v}$ for a fixed choice of Frobenius lift, $k(v) = \cO_{F_v} / (\varpi_v)$ for the residue field, and $q_v = \# k(v)$ for the cardinality of the residue field. If $S$ is a finite set of finite places of $F$ then we write $F_S / F$ for the maximal subextension of $\overline{F}$ unramified outside $S$ and $G_{F, S} = \Gal(F_S / F)$.

By definition, a number field $F$ is a CM field if it is totally complex and there exists an involution $c \in \Aut(F)$ such that for any embedding $\tau : F \to \mathbb{C}$, we have $\tau(c(z)) = \overline{\tau(z)}$ for all $z \in F$. If $F$ is a CM field, then we write $F^+ = F^{c=1}$ for its maximal totally real subfield. 

If $p$ is a prime, then we call a coefficient field a finite extension $E / \bbQ_p$ contained inside our fixed algebraic closure $\overline{\bbQ}_p$, and write $\cO$ for the valuation ring of $E$, $\varpi \in \cO$ for a fixed choice of uniformizer, and $k = \cO / (\varpi)$ for the residue field. 

Let $K$ be a non-archimedean characteristic $0$ local field, and let $\Omega$ 
be an algebraically 
closed field of characteristic $0$. We write $W_K \subset G_K$ for the Weil group of $K$ and $I_K \subset W_K$ for the inertia subgroup. We use the cohomological normalisation of 
class field theory: it is the isomorphism $\Art_K: K^\times \to W_K^{ab}$ which 
sends uniformizers to geometric Frobenius elements. We use the Tate 
normalisation of the local Langlands 
	correspondence 
for $\GL_n$: it is the bijection $\rec_K^T$ between isomorphism classes of irreducible, 
admissible $\Omega[\GL_n(K)]$-modules and isomorphism classes of Frobenius-semisimple Weil--Deligne 
representations over $\Omega$ of rank $n$ which is normalised as in \cite[\S 
2.1]{Clo14}. 

If $F$ is a number field and $\pi$ is an automorphic representation of 
$\GL_n(\bbA_F)$, we say that $\pi$ is regular algebraic if $\pi_\infty$ has the 
same infinitesimal character as an irreducible algebraic representation of 
$\Res_{F/\bbQ}\GL_n$. If $F$ is a CM field and $\pi$ is a regular algebraic, cuspidal, automorphic 
representation of $\GL_n(\bbA_F)$, then for any isomorphism $\iota: \overline{\mathbb{Q}}_p \to \mathbb{C}$ there is an 
associated Galois representation $r_{\pi, \iota} : G_F \to \GL_n(\overline{\mathbb{Q}}_p)$ (see e.g.\ \cite{Har16}). 

\section{Preliminaries in commutative algebra}
\begin{lemma}\label{lem_submodule_of_O_module}
    Let $E$ be a coefficient field with valuation ring $\cO$, and let $M$ be an $\cO$-module that can be generated by $r \geq 0$ elements. Then any $\cO$-submodule of $M$ can also be generated by $r$ elements.
\end{lemma}
\begin{proof}
    The minimal number of generators of $M$ is $a + b$, where $a = \dim_E M[1/p]$ and $b = \dim_k M[\varpi]$. If $N \leq M$ is an $\cO$-submodule, then $N[1/p] \leq M[1/p]$ and $N[\varpi] \leq M[\varpi]$, so the result follows from the inequality $a + b \leq r$.
\end{proof}
\begin{lemma}\label{lem_product_of_flat_is_flat}
    Let $R$ be a Noetherian ring, and let $(M_i)_{i \in I}$ be an indexed list of flat $R$-modules. Then $\prod_{i \in I} M_i$ is a flat $R$-module.
\end{lemma}
\begin{proof}
    See \cite[\href{https://stacks.math.columbia.edu/tag/05CZ}{Proposition 05CZ}]{stacks-project}. 
\end{proof}
\begin{lemma}\label{lem_inverse_limit_of_flat_is_flat}
    Let $R$ be a complete Noetherian local ring. Then:
    \begin{enumerate}
        \item Suppose that $R$ is Artinian, and that $(M_m)_{m \geq 1}$ is a projective system of flat $R$-modules with surjective transition maps. Then $M = \varprojlim M_m$ is a flat $R$-module, and for any ideal $J \leq R$, $M / J M \cong \varprojlim M_m / J M_m$. 
        \item Let $I \leq R$ be an ideal, and let $(M_m)_{m \geq 1}$ be a projective system of $R$-modules such that $M_m$ is a flat $R / I^m$-module, and the transition maps induce isomorphisms $M_{m+1} / I^m M_{m+1} \cong M_m$. Then $M = \varprojlim_m M_m$ is a flat $R$-module, and $M / I^m M \cong M_m$ for any $m \geq 1$. Moreover, for any ideal $J \leq R$, $M / J M \cong \varprojlim M_m / J M_m$. 
    \end{enumerate}
\end{lemma}
\begin{proof}
    This is \cite[Lemma 4.4.3]{Pan22}. 
\end{proof}
\begin{lemma} \label{lem:filtration_generator_bound}
    Let $N \geq 1$ be an integer, and let $H$  be an $\mathcal{O}[x]$-module with a filtration by $\mathcal{O}[x]$-submodules
    \[
        0 = F^N \leq F^{N-1} \leq  \cdots \leq F^0 = H.
    \]
    If $x (F^{i}/F^{i+1})$ can be generated as an $\mathcal{O}$-module by $r_i$ elements for each $i = 0, \dots, N-1$, then $x^N H$ can be generated as an $\mathcal{O}$-module by $r = \sum_{i=0}^{N-1} r_i$ elements. 
\end{lemma}

\begin{proof}
    For each $i = 0, \dots, N-1$, let $e_1^{(i)}, \dots, e_{r_i}^{(i)} \in F^{i}$ be lifts of a generating set of $x (F^{i}/F^{i+1})$. For any $h \in H$ we can find $a_j^{(0)} \in \mathcal{O}$ such that $h_1 = x h - \sum_{j=1}^{r_0} a_j^{(0)} e_{j}^{(0)} \in F^{1}$. Repeating the same argument with $h_1$, we eventually find $a_j^{(i)} \in \mathcal{O}$ such that
    \[
        x^N h = \sum_{i = 0}^{N-1} x^{N-1-i} \sum_{j = 1}^{r_i} a_{j}^{(i)} e_j^{(i)}.
    \]
    It follows that $x^N h$ is contained in the $\cO$-submodule of $H$ generated by the set $\{ x^{N-1-i} e_j^{(i)} \}$ of size $r = \sum r_i$. To conclude, we apply Lemma \ref{lem_submodule_of_O_module}. 
\end{proof}
\begin{definition}\label{def_bounded_inverse_system}
    We say that a sequence $(H_m)_{m \geq 1}$ of $\mathcal{O}$-modules is bounded if there exists a constant $r$ and surjections $\mathcal{O}^r \to H_m$ for all $m$.
\end{definition}

\begin{lemma} \label{lem:bounded_sequence_limit}
    Let $(H_m){m \geq 1}$ be a bounded inverse system  $\mathcal{O}$-modules such that $\varpi^m H_m = 0$ for all $m$. Then:
    \begin{enumerate}
        \item The limit $\varprojlim_m H_m$ is a finite $\mathcal{O}$-module.
        \item The module $R^1 \varprojlim_m H_m$ vanishes.
    \end{enumerate}
\end{lemma}

\begin{proof}
    Note that each $H_m$ has finite length. Thus, $(H_m)_{m \geq 1}$ is a Mittag-Leffler system. This implies that $R^1 \varprojlim_m H_m = 0$. Moreover, let 
    \[
        H'_m := \bigcap_{k \geq m} \im(H_k \to H_m).
    \]
    Then the natural map $\varprojlim H'_m \to \varprojlim H_m$ is an isomorphism and the transition maps $H'_{m+1} \to H'_m$ are surjective.

    There exists an integer $m_0$ such that $H'_m/(\varpi) \to H'_{m_0}/(\varpi)$ is an isomorphism for all $m \geq m_0$ (as the sequence $(H'_m / (\varpi))_m$ is a bounded sequence of $\cO / (\varpi)$-modules with surjective transition maps).
    Let $\phi \colon \mathcal{O}^r \to H'_{m_0}$ be a surjection and choose compatible 
    lifts $\phi_m \colon \mathcal{O}^r \to H'_m$ for all $m \geq m_0$. Then the limit of the $\phi_m$ yields a surjection
    \[
        \mathcal{O}^r \to \varprojlim_{m \geq m_0} H'_m \cong \varprojlim_{m} H'_m \cong \varprojlim_m H_m. \qedhere
    \]
\end{proof}
\begin{lemma}\label{lem_patching_lemma}
    Suppose given integers $g, r, l_0, q_0 \geq 0$, along with the following data:
    \begin{enumerate}
        \item a perfect complex of $S$-modules $M_{\infty,\bullet}$, where $S = E \llbracket Y_1, \dots, Y_r \rrbracket$;
        \item a commutative local $S$-subalgebra $T_\infty \subset \End_{\mathbf{D}(S)}(M_{\infty,\bullet})$, whose maximal ideal annihilates $H_{*}(M_\bullet)$, where $M_\bullet := M_{\infty, \bullet} \otimes_{S}^{\mathbb{L}} S / \mathfrak{m}_S$,
        \item a nilpotent ideal $I_\infty < T_\infty$ and a surjective local ring homomorphism $\Phi \colon E\llbracket x_1, \dots, x_g \rrbracket \to T_\infty/I_\infty$.
    \end{enumerate}
    If $H_*(M_\bullet)$ is non-zero, concentrated in degrees $[q_0, q_0 + l_0]$, and $g + l_0 \leq r$, then $\Phi$ is an isomorphism and $\Phi^{-1}(\mathfrak{m}_S T_\infty + I_\infty) = (x_1, \dots, x_g)$. 
\end{lemma}

\begin{proof}
    Since the Krull dimension is invariant under nilpotent thickenings, we have $\dim T_\infty \leq g$. Moreover, since $M_{\infty,\bullet}$ is perfect, $T_\infty$ is a finite $S$-algebra. Then \cite[Lemma 2.8]{Kha17} implies that
    \[
        \dim_S H_{*}(M_{\infty, \bullet}) = \dim_{T_\infty} H_{*}(M_{\infty, \bullet}) \leq g.
    \]
    By the Calegari--Geraghty lemma \cite[Lemma 2.9]{Kha17} we also have
    \[
        \dim_S H_{*}(M_{\infty, \bullet}) \geq \dim S - l_0 \geq g
    \] 
    and the previous equation implies that $\dim_S H_*(M_{\infty, \bullet}) = g$. Thus, the second part of \cite[Lemma 2.9]{Kha17}
    shows that $M_{\infty, \bullet}$ is quasi-isomorphic to $H_{q_0}(M_{\infty, \bullet})[q_0]$ and $H := H_{q_0}(M_{\infty, \bullet})$ has projective dimension $l_0$. 

    Therefore, the Auslander--Buchsbaum formula implies that
    $\depth_S H = \dim S - l_0 = g$. It follows from 
    \cite[\href{https://stacks.math.columbia.edu/tag/0AUK}{Lemma 0AUK}]{stacks-project} that 
    \[
    \depth_{T_\infty} H \geq \depth_S H = g.
    \]
    It follows that $\depth_{T_\infty} H = \dim T_\infty = g$. Since $\dim T_\infty / I_\infty = g$, it must be the case that $\Phi$ is an isomorphism. Since $H / I_\infty H$ has $g$-dimensional support, and $T_\infty / I_\infty$ is a domain, it follows that $H / I_\infty H$ is a faithful $T_\infty / I_\infty$-module. 

    To conclude, we will apply \cite[Lemma 2.5.9]{AC23}, which implies that if $\mathfrak{m}_S (H / I_\infty H) = \mathfrak{m}_{T_\infty} (H / I_\infty H)$, then $\mathfrak{m}_S (T_\infty / I_\infty) = \mathfrak{m}_{T_\infty} (T_\infty / I_\infty) = (\Phi(x_1), \dots, \Phi(x_g)) (T_\infty / I_\infty)$. The inclusion \[ \mathfrak{m}_S (H / I_\infty H) \leq \mathfrak{m}_{T_\infty} (H / I_\infty H)\]
    is immediate, since $\mathfrak{m}_S \leq \mathfrak{m}_{T_\infty}$. On the other hand, we have $H / \mathfrak{m}_S H = H_{q_0}(M_\bullet)$, and we are assuming that $\mathfrak{m}_{T_\infty} H_{q_0}(M_\bullet) = 0$, hence 
    \[ \mathfrak{m}_{T_\infty} H \leq \mathfrak{m}_S H, \]
    hence 
    \[ \mathfrak{m}_{T_\infty} (H / I_\infty H) \leq \mathfrak{m}_S (H / I_\infty H). \]
    This completes the proof. 
\end{proof}

\section{Pseudodeformation rings}\label{sec_galois_theory}

Fix a prime $p$, a coefficient field $E / \mathbb{Q}_p$, an integer $n \geq 1$, and a CM number field $F$. Let $\mathcal{C}_\cO$ denote the category of complete Noetherian local $\cO$-algebras with residue field $k$. 

Let $S$ be a finite set of finite places of $F$, containing the $p$-adic places $S_p$. Suppose given a continuous representation $\overline{\rho} : G_{F, S} \to \GL_n(k)$, and let $\overline{D}$ denote its group determinant. We write $R_{\overline{D}, S} \in \mathcal{C}_\cO$ for the object representing the functor $\Def_{\overline{D}, S}$ of pseudodeformations lifting $\overline{D}$ (see \cite[Proposition 2.12]{New23}). We recall the following useful result, which is \cite[Lemma 2.13]{New23}:
\begin{lemma}\label{lem_uniform_bound_on_tangent_space}
    Let $q \geq 0$ be an integer. Then there exists an integer $g = g(q)$ such that for any finite set $Q$ of finite places of $F$ of cardinality $q$, there is a surjection $\cO \llbracket X_1, \dots, X_g \rrbracket \to R_{\overline{D}, S \cup Q}$ in $\mathcal{C}_\cO$.
\end{lemma}

\para Fix integers $a \leq b$. Following \cite[\S 2.3]{New23}, we write $\Def^{[a, b]}_{\overline{D}, S} \subset \Def_{\overline{D}, S}$ for the subfunctor which associates to any $A \in \mathcal{C}_\cO$ the set of pseudodeformations $D$ of $G_{F, S}$ over $A$ satisfying the following condition:
\begin{itemize}
    \item There exists a Cayley--Hamilton representation $(\cO[G_{F, S}], D) \to (B, D')$ over $A$ such that for each $m \geq 1$, $B / \mathfrak{m}_A^m B$, equipped with its left $G_{F, S}$-action, has the following property: for each $v \in S_p$, $B / \mathfrak{m}_A^m B$ is isomorphic as $\mathbb{Z}_p[G_{F_v}]$-module to a subquotient of a $\mathbb{Z}_p$-lattice in a semistable $\mathbb{Q}_p[G_{F_v}]$-module with Hodge--Tate weights contained in the interval $[a, b]$. 
\end{itemize}
Using the results of \cite{WWE19}, \cite[Proposition 2.14]{New23} shows that $\Def^{[a, b]}_{\overline{D}, S}$ is represented by a quotient $R_{\overline{D}, S}^{[a, b]}$ of $R_{\overline{D}, S}$. We write $I_S^{[a, b]} \leq R_{\overline{D}, S}$ for the kernel of the corresponding surjection. Thus, for any homomorphism $R_{\overline{D}, S} \to A$ in $\mathcal{C}_\cO$ corresponding to a pseudodeformation $D$ over $A$, we see that $D$ satisfies the above semistability condition if, and only if, $I_S^{[a, b]} A = 0$.

Now suppose given a lift $\rho : G_{F, S} \to \GL_n(\cO)$ of $\overline{\rho}$ such that $\rho \otimes_\cO E$ is absolutely irreducible, and such that for each $v \in S_p$, $\rho|_{G_{F_v}} \otimes_\cO E$ is semistable with Hodge--Tate weights in $[a, b]$. If $N \geq 1$, we call a set $Q$ of finite places of $F$ a Taylor--Wiles set of level $N$ if it satisfies the following conditions:
\begin{itemize}
    \item $Q \cap S = \emptyset$.
    \item For each $v \in Q$, $q_v \equiv 1 \text{ mod }p^N$.
    \item For each $v \in Q$, the characteristic polynomial $\det(X - \rho(\Frob_v))$ splits in $\cO[X]$ as a product of distinct linear factors. 
\end{itemize}
We call a Taylor--Wiles datum of level $N$ a tuple $(Q, (\alpha_{v, 1}, \dots, \alpha_{v, n})_{v \in Q})$, where $Q$ is a Taylor--Wiles set of level $N$ and, for each $v \in Q$, $\alpha_{v, 1}, \dots, \alpha_{v, n}$ is an ordering of the roots of $\det(X - \rho(\Frob_v))$ in $\cO$. 

If $Q$ is a Taylor--Wiles set, there is a tautological surjection $R_{\overline{D}, S \cup Q} \to R_{\overline{D}, S}$, and we have $I_{S \cup Q}^{[a, b]} R_{\overline{D}, S} = I_S^{[a, b]}$. Let us write 
$\mathfrak{p}_{S \cup Q} \leq R_{\overline{D}, S \cup Q}$ for the kernel of the homomorphism $R_{\overline{D}, S \cup Q} \to \cO$ associated to the group determinant $D$ of $\rho$ (which factors through $R_{\overline{D}, S \cup Q}^{[a, b]}$, because $\rho$ is assumed to be semistable). Our goal is to show that we can choose $Q$ such that the relative tangent space $\mathfrak{p}_{S \cup Q} / (I_{S \cup Q}^{[a, b]}, \mathfrak{p}_{S \cup Q}^2)$ is controlled. We will be able to do this only under certain conditions on $\rho$, that we now recall. First, a local condition:
\begin{definition}
    Let $v$ be a finite place of $F$, and let $(r, N) = \mathrm{WD}(\rho|_{G_{F_v}} \otimes_\cO E)$. We say that $\rho|_{G_{F_v}}$ is generic if there is no non-zero morphism $(r, N) \to (r(1), N)$ of Weil--Deligne representations over $E$.
\end{definition}
By \cite[Remark 1.2.9]{All16}, this is equivalent to the equality $H^1_f(F_v, \ad \rho \otimes_\cO E) = H^1_g(F_v, \ad \rho \otimes_\cO E)$ of local Selmer groups \`a la Bloch--Kato. 

Next, a global condition, which is the same as \cite[Definition 2.23]{New23}:
\begin{definition}
    Let $H \leq \GL_n(\cO)$ be a compact subgroup. We say that $H$ is enormous if for every simple $E[H]$-submodule $U$ of the adjoint representation on $M_{n \times n}(E)$, we can find $h \in H$ with $n$ distinct eigenvalues in $\cO$ such that $\tr E[h] \cdot U \neq 0$.
\end{definition}
Here is the statement we will use.
\begin{proposition}\label{prop_existence_of_TW_primes}
    Let $q \geq \operatorname{corank} H^1(F_S / F, \ad \rho \otimes_\cO E / \cO(1))$, and suppose that $F$ contains an imaginary quadratic field $F_0$. Suppose that $\rho$ satisfies the following conditions:
    \begin{enumerate}
    \item For every place $v \in S$, $\rho|_{G_{F_v}}$ is generic. 
    \item Let $L / F(\zeta_{p^\infty})$ be the extension cut out by $\ad \rho$. Then $H^1(L / F, \ad \rho(1) \otimes_\cO E) = 0$.
    \item The image $\rho(G_{F(\zeta_{p^\infty})})$ is enormous.
    \end{enumerate}
    Then $q \geq n [ F^+ : \mathbb{Q} ]$, 
    and there exists $d \geq 1$ with the following property: for any $N \geq 1$, we can find a Taylor--Wiles datum $(Q_N, (\alpha_{v, 1}, \dots, \alpha_{v, n})_{v \in Q_N})$ of level $N$ with $|Q_N| = q$ and $\sum_{v \in Q_N} \sum_{i \neq j} \ord_{\varpi} (\alpha_{v, i} - \alpha_{v, j}) \leq d$, and a map
    \[ \cO \llbracket X_1, \dots, X_{n q - n[F^+ : \mathbb{Q}] + 1} \rrbracket \to R_{\overline{D}, S \cup Q_N} \]
    such that $(X_1, \dots, X_{nq- n [ F^+ : \mathbb{Q} ] + 1})  R_{\overline{D}, S \cup Q_N} \leq \mathfrak{p}_{S \cup Q_N}$, and 
    \[ \mathfrak{p}_{S \cup Q_N} / (\mathfrak{p}_{S \cup Q_N}^2, I_{S \cup Q_N}^{[a, b]}, x_1, \dots, x_{nq -  n [ F^+ : \mathbb{Q} ] + 1}) \]
    is a quotient of $(\cO / \varpi^d \cO)^{g(q)}$ (where $g = g(q)$ is as defined in the statement of Lemma \ref{lem_uniform_bound_on_tangent_space}). Moreover, the residue characteristic of each $v \in Q_N$ is split in $F_0$.
\end{proposition}
\begin{proof}
    The proof is essentially the same as the proof of \cite[Corollary 2.27]{New23}, with the following modifications:
    \begin{itemize}
        \item We no longer assume purity, which is used in \emph{loc. cit.} to deduce the vanishing of the group $H^1(L / F, \ad \rho(1) \otimes_\cO E)$. Instead, we have inserted this vanishing as a hypothesis.
        \item We no longer assume oddness, as we are working over $F$ without imposing a conjugate self-duality condition. This has the consequence that the number of generators $nq$ is replaced by $nq -  n [ F^+ : \mathbb{Q} ] + 1$.
        \item Choosing primes with the Chebotarev density theorem as in \emph{loc. cit.}, we can find a set $Q_N$ annihilating the dual Selmer group, with $|Q_N| = \operatorname{corank} H^1(F_S / F, \ad \rho \otimes_\cO E / \cO(1))$, and such that for each $v \in Q_N$, the residue characteristic is split in $F_0$ (because the set of places of $F$ split over $\mathbb{Q}$, and therefore with residue characteristic split in $F_0$, has Dirichlet density 1). \qedhere
    \end{itemize} 
\end{proof}

\begin{remark} We note that although purity is a convenient criterion for the vanishing of the cohomology group appearing in the statement of Proposition \ref{prop_existence_of_TW_primes}, this can often be checked directly (cf. \cite[Lemma 4.28]{All23}).
\end{remark}
\section{Patching}\label{sec_patching}

In this section we will prove our main vanishing result under additional assumptions, before deducing the general case in \S \ref{sec_general_results} below. Let $F$ be a CM field, let $n \geq 1$ be an integer, and let $\pi$ be a cuspidal, regular algebraic automorphic representation of $\GL_n(\mathbb{A}_F)$. We will prove:
\begin{theorem}\label{thm_vanishing_of_Selmer_special_case}
    Fix a prime $p$, an isomorphism $\iota : \overline{\mathbb{Q}}_p \to \mathbb{C}$,  a coefficient field $E / \mathbb{Q}_p$, and a representation $\rho : G_F \to \GL_n(\cO)$ that is conjugate, in $\GL_n(\overline{\bbQ}_p)$, to $r_{\pi, \iota}$. Suppose that the following conditions are satisfied:
    \begin{enumerate}
      \item For each finite place $v$ of $F$, $\pi_v$ is Iwahori-spherical.
      \item For each finite place $v$ of $F$ such that $\pi_v$ is ramified or such that $v | p$, $\rho|_{G_{F_v}}$ is generic. 
      \item Let $L / F(\zeta_{p^\infty})$ denote the extension cut out by $\ad \rho$. Then \[
      H^1(L / F, \ad \rho(1) \otimes_\cO E) = 0.
      \]
         \item $\rho(G_{F(\zeta_{p^\infty})})$ is enormous.
         \item For every rational prime $l$ such that $l = p$ or $\pi$ is ramified at an $l$-adic place of $F$, there exists an imaginary quadratic subfield $F_0 \leq F$ such that $l$ splits in $F_0$. 
    \item For each $p$-adic place $a$ of $F$, lying over the place $\overline{a}$ of $F^+$, there exists a $p$-adic place $\overline{b} \neq \overline{a}$ of $F^+$ such that
    \[ \sum_{\overline{v} \neq \overline{a}, \overline{b}} [F^+_{\overline{v}} : \mathbb{Q}_p ] > \frac{1}{2} [ F^+ : \mathbb{Q} ] \]
    (where the sum runs over the set of $p$-adic places $\overline{v}$ of $F^+$ not equal to $\overline{a}$ or $\overline{b}$).
    \end{enumerate}
    Then $H^1_f(F, \ad \rho \otimes_\cO E) = 0$.
\end{theorem}

\para We briefly describe the role of each hypothesis. Hypothesis (1) is a simplifying assumption that is always satisfied under base change. Hypothesis (2) is required in order for the local lifting rings to be smooth at the point determined by $\rho|_{G_{F_v}}$; it is conjectured to always hold for geometric Galois representations. Hypotheses (3) and (4) enter into the Taylor--Wiles argument (in the form applied here). Hypotheses (5) and (6) are used in the proof of local-global compatibility, a basic assumption for our purposes.  

\para We are free to enlarge $E$, so we can assume that $E$ contains the eigenvalues of every element in the image of $\rho$ (using that $E$, being a $p$-adic local field, has finitely many extensions of bounded degree).

We begin the proof of the theorem by introducing some of the necessary objects. We set $G = \Res_{F / \mathbb{Q}} \GL_n$, and write $X^G$ for the symmetric space associated with $G$. We say that an open compact subgroup $K = \prod_v K_v \leq \GL_n(\mathbb{A}_F^\infty)$ is good if it satisfies the following conditions:
\begin{itemize}
    \item $K$ is neat.
    \item For each finite place $v$ of $F$, $K_v \leq \GL_n(\cO_{F_v})$. 
\end{itemize}
\begin{lemma}\label{lem_genericity_at_unramified_places}
    Let $u \nmid p$ be a place of $F$ such that $\pi_u$ is unramified. Then $\rho|_{G_{F_u}}$ is unramified and generic.
\end{lemma}
\begin{proof}
    It follows from the main theorem of \cite{Var24} that $\rho|_{G_{F_u}}$ is unramified. The classification of generic unitary representations of $\GL_n(F_u)$ \cite{Tad86} shows that $\rho|_{G_{F_{u}}}$ is generic.
\end{proof}

\para Using Lemma \ref{lem_genericity_at_unramified_places}, we may choose a finite set $S$ of finite places of $F$ satisfying the following conditions:
\begin{itemize}
    \item $S$ contains every place $v$ of $F$ such that $v | p$ or $\pi_v$ is ramified.
    \item $S$ is stable under complex conjugation.
    \item For each place $v \in S$, there exists an imaginary quadratic subfield $F_0 \subset F$ such that the residue characteristic of $v$ splits in $F_0$.
    \item For each place $v \in S$, $\rho|_{G_{F_v}}$ is generic. 
    \item There exists a place $u \in S$ such that $u \nmid p$, $\pi_u$ is unramified, and if $\ell_u$ denotes the residue characteristic of $u$, then $\ell_u$ is odd and $\ell_u$ is unramified in $F$.
\end{itemize}
We also fix a rational prime $\ell$ split in $F$, and prime to each element of $S$. After twisting $\pi$ by a power of the norm character, we can assume that, for each place $v | \ell$ of $F$, the representation $\rho|_{G_{F_v}} \oplus (\rho|_{G_{F_v}}^{c, \vee} \otimes \epsilon^{1-2n})$ is generic. (This will ensure that the Hecke operator $\delta_\ell$ appearing in the statement of Theorem \ref{thm_existence_of_Galois_action} below acts on $\pi^\infty$ with a non-zero eigenvalue.) We take $K = \prod_v K_v$ to be the good subgroup of $\GL_n(\mathbb{A}_F^\infty)$ defined as follows: 
\begin{itemize}
    \item If $v \not\in S$, then $K_v = \GL_n(\cO_{F_v})$.
    \item If $v \in S - \{ u \}$, then $K_v = \Iw_v$ is the standard Iwahori subgroup of $\GL_n(\cO_{F_v})$ (matrices which are upper-triangular modulo $\varpi_v$).
    \item $K_u = \ker(\GL_n(\cO_{F_u}) \to \GL_n(k(u)))$. (Then the conditions on $u$ imply that $K_u$ has no non-trivial elements of finite order.)
\end{itemize}
If $( Q = (v_1, \dots, v_q), (\alpha_{v, 1}, \dots, \alpha_{v, n})_{v \in Q})$ is a Taylor--Wiles datum, prime to $S$ and $\ell$, then we write $K_1(Q) \leq K_0(Q) \leq K$ for the good subgroups defined as follows:
\begin{itemize}
    \item $K_1(Q)_v = K_0(Q)_v = K_v$ if $v \not\in Q$.
    \item If $v \in Q$, then $K_0(Q)_v = \Iw_v$, and $K_1(Q)_v = \Iw_{v, 1} := \ker(\Iw_v \to k(v)^\times(p)^n)$ is its maximal subgroup of $p$-power index. (Thus $k(v)^\times(p)$ denotes the maximal $p$-power quotient of $k(v)^\times$.)
\end{itemize}
Then we can naturally identify $K_0(Q) / K_1(Q) = \Delta_Q := \prod_{v \in Q} k(v)^\times(p)^n$. We will always impose the further condition on $Q$ that the residue characteristic of each place $v \in Q$ is split in some imaginary quadratic subfield of $F$.

We write $\mathbb{T}^S = \mathcal{H}(\GL_n(\mathbb{A}_F^{\infty, S}), \GL_n(\widehat{\cO}_F^S)) \otimes_{\mathbb{Z}} \cO$ for the usual unramified Hecke algebra. If $Q$ is a Taylor--Wiles datum, then we define $A_Q = \cO[ \{ (t_{v}^{(j)})^{\pm 1} \}_{v\in Q, 1 \leq j \leq n}]$. According to \cite[Proposition 2.2.7]{10author}, there are universal characters 
\[ t_{v, i} : F_v^\times \to \mathcal{H}(\GL_n(F_v), \Iw_{v, 1}) \otimes_{\mathbb{Z}} \cO, \]
and we use the values $t_{v, i}(\varpi_v)$ to endow $\mathcal{H}(\GL_n(F_v), \Iw_{v, 1}) \otimes_{\mathbb{Z}} \cO$ with an $A_Q[\Delta_Q]$-algebra structure. On the other hand, the group $W_Q = \prod_{v \in Q} S_n$ acts on $A_Q$, and we can identify $\mathbb{T}^S = \mathbb{T}^{S \cup Q} \otimes_{\cO} A_Q^{W_Q}$. We write $\delta_Q = \prod_{v \in Q} \prod_{1 \leq i < j \leq n} (t_{v, i} - t_{v, j})^2 \in A_Q^{W_Q}$. 

We also introduce an `abstract' version of these objects. Fixing an integer $q \geq 1$, we set $A = \cO[\{ (t_{i}^{(j)})^{\pm 1} \}_{1 \leq i \leq q, 1 \leq j \leq n}]$, $W = \prod_{j=1}^q S_n$, and write $\delta = \prod_{j=1}^q \prod_{1 \leq i < k \leq n} (t_{j, i} - t_{j, k})^2 \in A^W$.

Let $\lambda \in (\mathbb{Z}^n_+)^{\Hom(F,E)}$ be the weight of $\pi$ (or rather, its image under $\iota^{-1}$), and let $\cV_\lambda$ be the $\cO[\prod_{v | p}\GL_n(\cO_{F_v})]$-module defined in \cite[\S 2.2.1]{10author} (it is an $\cO$-lattice in the representation of $(\Res_{F / \mathbb{Q}} \GL_{n, F})_E$ of highest weight $\lambda$).  Let $\mathfrak{S}_\bullet$ be the complex of singular chains  of the topological space
\[ \mathfrak{X}_G = \GL_n(F) \backslash X^G \times \GL_n(\mathbb{A}_F^\infty), \]
with $\cO$-coefficients. Then $\mathfrak{S}_\bullet$ is a complex of $\cO[\GL_n(\mathbb{A}_F^\infty)]$-modules, free as $\cO$-modules. We define
\begin{align*}
    C_\bullet &= \mathfrak{S}_\bullet \otimes_{\cO[K]} \mathcal{V}_\lambda^\vee, \\
 C_{Q, 0, \bullet} &=\mathfrak{S}_\bullet \otimes_{\cO[K_0(Q)]} \mathcal{V}_\lambda^\vee, 
\end{align*}
and
\[ C_{Q, 1, \bullet} = \mathfrak{S}_\bullet \otimes_{\cO[K_1(Q)]} \mathcal{V}_\lambda^\vee. \]
These complexes have the following desirable properties.
\begin{proposition}\label{prop_complexes_compute_classical_cohomology} For an open compact subgroup $K_0 \leq K$, let $X^G_{K_0} = \mathfrak{X}_G / K_0$ denote the quotient topological space, and write again $\mathcal{V}_\lambda^\vee$ for the associated local system on $X^G_{K_0}$. Then:
    \begin{enumerate}
        \item $C_\bullet$ is a bounded above complex of $\mathbb{T}^S$-modules, free as $\cO$-modules, equipped with a canonical isomorphism $H_\ast(C_{\bullet}) = H_\ast(X^G_K, \mathcal{V}_\lambda^\vee)$.
        \item Let $Q$ be a Taylor--Wiles set. Then: 
        \begin{enumerate}
        \item $C_{Q, 0, \bullet}$ is a bounded above complex of $\mathbb{T}^S$-modules, free as $\cO$-modules, equipped with a canonical isomorphism $H_\ast(C_{Q, 0, \bullet}) = H_\ast(X^G_{K_0(Q)}, \mathcal{V}_\lambda^\vee)$. 
        \item $C_{Q, 1, \bullet}$ is a bounded above complex of $\mathbb{T}^S[\Delta_Q]$-modules, free as $\cO[\Delta_Q]$-modules, equipped with a canonical isomorphism $H_\ast(C_{Q, 1, \bullet}) = H_\ast(X^G_{K_1(Q)}, \mathcal{V}_\lambda^\vee)$.
        \end{enumerate}
    \end{enumerate}
\end{proposition}
\begin{proof}
    See \cite[Proposition 6.2]{Kha17} for a very similar statement. In 2(a) and (b), we are using the identification $\mathbb{T}^S = \mathbb{T}^{S \cup Q} \otimes_\cO A_Q^{W_Q}$ to define the action of $\mathbb{T}^S$ on each complex. 
\end{proof}

\para We will patch these complexes using ultrafilters. Suppose therefore given an infinite sequence $(Q_N, ( \alpha_{v, i}))$ of Taylor--Wiles data such that for each $N \geq 1$, $Q_N$ has level $N$, and $|Q_N| = q$. Let $S_\infty = \cO \llbracket (\mathbb{Z}_p^n)^q \rrbracket$, and fix for each $N \geq 1$ surjections $(\mathbb{Z}_p^n)^q \to \Delta_{Q_N}$. We write $\mathfrak{a} \leq S_\infty$ for the augmentation ideal. 

Let $\mathcal{F}$ be a non-principal ultrafilter on $\mathbb{N}$, and let $\mathbf{R} = \prod_{N \geq 1} \cO$. If $I \in \mathcal{F}$, then we define $e_I \in \mathbf{R}$ by $e_{I, N} = 1$ if $N \in I$, and $e_{I, N} = 0$ otherwise. Then $\mathcal{S} = \{ e_I \mid I \in \mathcal{F} \}$ is a multiplicative subset of $\mathbf{R}$, and we can define the localisation $\mathbf{R}_{\mathcal{F}} = \mathcal{S}^{-1} \mathbf{R}$. The natural map $\mathbf{R} \to \mathbf{R}_{\mathcal{F}}$ is surjective, and factors through the projection $\prod_{N \geq 1} \cO \to \prod_{N \geq m} \cO$ for any $m \geq 1$. 

Let $\mathfrak{m} = \ker( \mathbb{T}^S \to k)$ denote the maximal ideal associated to the reduction modulo $\varpi$ of the Hecke eigenvalues of $\iota^{-1} \pi^\infty$. We define:
\[ M'(m)_\bullet = \mathbf{R}_{\mathcal{F}} \otimes_{\mathbf{R}} \prod_{r \geq 1} \left( C_{\bullet} \otimes_{\mathbb{T}^S} \mathbb{T}^S_{\mathfrak{m}} / (\varpi^m) \right) ; \]
\[ M(m)_\bullet =  \mathbf{R}_{\mathcal{F}} \otimes_{\mathbf{R}}  \prod_{r \geq 1} \left( C_{ \bullet}\otimes_{\mathbb{T}^S} \mathbb{T}^S_{\mathfrak{m}} / (\varpi^m) \right) \otimes_{A^W} A \]
(where $A^W$ acts on $\prod_{r \geq 1} C_\bullet\otimes_{\mathbb{T}^S} \mathbb{T}^S_{\mathfrak{m}} / (\varpi^m)$ via the diagonal map $A^W \to \prod_{r \geq 1} A_{Q_r}^{W_{Q_r}}$);
\[ 
M_{0}(m)_{\bullet} =  \mathbf{R}_{\mathcal{F}} \otimes_{\mathbf{R}} \prod_{r \geq m} \left( C_{Q_r, 0, \bullet}\otimes_{\mathbb{T}^S} \mathbb{T}^S_{\mathfrak{m}} / (\varpi^m) \right); 
\]
and, if $m \geq N \geq 1$, 
\[ M_{1, N}(m)_{\bullet} = \mathbf{R}_{\mathcal{F}} \otimes_{\mathbf{R}} \prod_{r \geq m + N} \left( C_{Q_r, 1, \bullet} / (\mathfrak{a}^N) \otimes_{\mathbb{T}^S} \mathbb{T}^S_{\mathfrak{m}} / (\varpi^m ) \right). \] 
Suppose that $r \geq m$. Then, in \cite[Proposition 3.1]{New23} are constructed morphisms of $\mathbb{T}^S \otimes_{A_{Q_r}^{W_{Q_r}}} A_{Q_r}$-modules
\[ \alpha_{Q_r} : C_{\bullet} \otimes_{\mathbb{T}^S} \mathbb{T}^S_{\mathfrak{m}} / (\varpi^m) \otimes_{A_{Q_r}^{W_{Q_r}}} A_{Q_r} \to C_{Q_r, 0, \bullet} \otimes_{\mathbb{T}^S} \mathbb{T}^S_{\mathfrak{m}} / (\varpi^m) \]
and
\[ \beta_{Q_r} : C_{Q_r, 0, \bullet} \otimes_{\mathbb{T}^S} \mathbb{T}^S_{\mathfrak{m}}/ (\varpi^m) \to C_{\bullet} \otimes_{\mathbb{T}^S} \mathbb{T}^S_{\mathfrak{m}} / (\varpi^m) \otimes_{A_{Q_r}^{W_{Q_r}}} A_{Q_r}. \]
More precisely, in \emph{loc. cit.} such morphisms are constructed for invariants, as opposed to the co-invariants we consider here. However, using the notation of \cite[\S 3]{New23}, we can take $\alpha_{Q_r}( (m_{\mathbf{a}} )_{\mathbf{a}}) = \sum_{\mathbf{a}} Q_{\mathbf{a}, 1} \tr_{K_0(Q_r) / K} m_{\mathbf{a}}$, $\beta_{Q_r} ( m ) = (\operatorname{cores}_{K_0(Q_r) / K} P_{1, \mathbf{a}} m)_{\mathbf{a}}$. The dual of the proof of \cite[Proposition 3.1]{New23} then shows that  $\alpha_{Q_r} \circ \beta_{Q_r}$ and $\beta_{Q_r} \circ \alpha_{Q_r}$ are both given by multiplication by the element $\delta_{Q_r}^{n!}$. 
\begin{proposition}\label{prop_patched_complexes_torsion_level}
\leavevmode
    \begin{enumerate}
        \item The natural map $C_\bullet \otimes_{\mathbb{T}^S} \mathbb{T}^S_{\mathfrak{m}}/ (\varpi^m) \to M'(m)_\bullet$ is a quasi-isomorphism of complexes of $\mathbb{T}^S$-modules, free as $\cO / (\varpi^m)$-modules. 
        \item $M_0(m)_\bullet$ is a complex of $\mathbb{T}^S \otimes_\cO A$-modules, free as $\cO / (\varpi^m)$-modules. 
        \item The products of these maps $\alpha_{Q_r}$, $\beta_{Q_r}$ ($r \geq m$) described above induce morphisms of complexes of $\mathbb{T}^S \otimes_\cO A$-modules
        \[ \alpha : M(m)_\bullet \to M_0(m)_\bullet, \]
        \[ \beta : M_0(m)_\bullet \to M(m)_\bullet, \]
        such that $\alpha \circ \beta$ and $\beta \circ \alpha$ are both given by multiplication by the element $\delta^{n!} \in A$. 
        \item $M_{1, N}(m)_\bullet$ is a complex of $\mathbb{T}^S \otimes_\cO A \otimes_\cO S_\infty / \mathfrak{a}^N$-modules, free as $S_\infty / (\varpi^m, \mathfrak{a}^N)$-modules. Moreover, there are canonical isomorphisms
        \[ M_{1, N+1}(m)_\bullet / (\mathfrak{a}^N) \cong M_{1, N}(m)_\bullet \]
        (for $N \geq 1$) and
        \[ M_{1, 1}(m)_\bullet / (\mathfrak{a}) \cong M_{0}(m)_\bullet \]
        of complexes of $\mathbb{T}^S \otimes_\cO A \otimes_\cO S_\infty$-modules. 
    \end{enumerate}
\end{proposition}
\begin{proof}
    (1) We compute
    \begin{align*}
        H_\ast( \mathbf{R}_\mathcal{F} \otimes_{\mathbf{R}} \prod_{r \geq 1} C_\bullet \otimes_{\mathbb{T}^S} \mathbb{T}^S_{\mathfrak{m}} / (\varpi^m)) &=  \mathbf{R}_\mathcal{F} \otimes_{\mathbf{R}} H_\ast( \prod_{r \geq 1} C_\bullet\otimes_{\mathbb{T}^S} \mathbb{T}^S_{\mathfrak{m}} / (\varpi^m)) \\
        &= \mathbf{R}_\mathcal{F} \otimes_{\mathbf{R}} \prod_{r \geq 1} H_\ast( C_\bullet/ (\varpi^m))_{\mathfrak{m}}.
    \end{align*} 
    Therefore, the fact that the natural map is a quasi-isomorphism follows from the fact that if $U$ is a finitely generated torsion $\cO$-module, then the localisation of the diagonal
    \[ U \to \mathbf{R}_\mathcal{F} \otimes_{\mathbf{R}} \prod_{r \geq 1} U \]
    is an isomorphism. To show that $M'(m)_\bullet$ is a complex of free $\cO / (\varpi^m)$-modules, we can apply Lemma \ref{lem_product_of_flat_is_flat} and \cite[\href{https://stacks.math.columbia.edu/tag/051G}{Lemma 051G}]{stacks-project}.

    (2) We make $C_{Q_r, 0, \bullet}$ into a $\mathbb{T}^S$-module using the isomorphism $\mathbb{T}^S \cong \mathbb{T}^{S \cup Q_r} \otimes_\cO A_{Q_r}^{W_{Q_r}}$ (in other words, by making the Hecke operators at $Q$ act through the Bernstein centre). We make $A$ act on $M_0(m)_\bullet$ via the $A_{Q_r}$-action on $C_{Q_r, 0, \bullet}$ and the diagonal map $A \to \prod_{r \geq 1} A_{Q_r}$. The freeness follows as for $M'(m)_\bullet$. 

    (3) The product induces morphisms
    \[ \mathbf{R}_{\mathcal{F}} \otimes_{\mathbf{R}} \prod_{r \geq 1} \left( C_\bullet \otimes_{\mathbb{T}^S} \mathbb{T}^S_{\mathfrak{m}} / (\varpi^m) \otimes_{A_{Q_r}^{W_{Q_r}}} A_{Q_r} \right) \rightleftarrows M_0(m)_\bullet \]
    whose composite in either direction is given by $\delta^{n!}$. (Here we remind the reader that, since the ultrafilter is non-principal, dropping finitely many terms from the product gives a result that is canonically isomorphic.) To give the claimed statement, we need to check that there is an isomorphism
    \begin{align*} &\mathbf{R}_{\mathcal{F}} \otimes_{\mathbf{R}} \prod_{r \geq 1} \left( C_\bullet\otimes_{\mathbb{T}^S} \mathbb{T}^S_{\mathfrak{m}} / (\varpi^m) \otimes_{A_{Q_r}^{W_{Q_r}}} A_{Q_r} \right)  \\
    \cong \  &  \mathbf{R}_{\mathcal{F}} \otimes_{\mathbf{R}} \left(  \prod_{r \geq 1} C_{ \bullet} \otimes_{\mathbb{T}^S} \mathbb{T}^S_{\mathfrak{m}}/ (\varpi^m) \right) \otimes_{A^W} A \end{align*}
    of complexes of $\mathbb{T}^S$-modules. This is true: since $A$ is a finite free $A^W$-module, we can take it inside the product. 

    (4) This is a simple consequence of the properties of $C_{Q_r, 1, \bullet}$ and Lemma \ref{lem_product_of_flat_is_flat}. Note that the condition $r \geq m + N$ implies that $S_\infty / (\mathfrak{a}^N, \varpi^m)$ is a quotient of $\cO[\Delta_{Q_r}]$.
\end{proof}

\para Finally, we define
\[ M'_\bullet  = \varprojlim_m M'(m)_\bullet, \]
\[ M_\bullet = \varprojlim_m M(m)_\bullet, \]
\[ M_{0, \bullet} = \varprojlim_m M_0(m)_\bullet, \]
and, if $N \geq 1$,
\[ M_{1, N, \bullet} = \varprojlim_m M_{1, N}(m)_\bullet. \]
\begin{proposition}\label{prop_patched_complexes_torsion_free_level}\leavevmode
    \begin{enumerate}
        \item There is an isomorphism $M'_\bullet \otimes_{A^W} A \to  M_\bullet$ of complexes of $\mathbb{T}^S \otimes_\cO A$-modules, flat as $\cO$-modules. Moreover, there is a quasi-isomorphism $C_\bullet \otimes_{\mathbb{T}^S} \mathbb{T}^S_{\mathfrak{m}} \to M'_\bullet$ of complexes of $\mathbb{T}^S$-modules.
        \item $M_{0, \bullet}$ is a complex of $\mathbb{T}^S \otimes_\cO A$-modules, flat as $\cO$-modules. Moreover, there are morphisms
        \[ \alpha : M_\bullet \to M_{0, \bullet}, \beta : M_{0, \bullet} \to M_\bullet \]
        of complexes of $\mathbb{T}^S$-modules, such that $\alpha \beta$ and $\beta \alpha$ are both equal to multiplication by $\delta^{n!}$.
        \item $M_{1, N, \bullet}$ is a complex of $\mathbb{T}^S \otimes_\cO A \otimes_\cO S_\infty / \mathfrak{a}^N$-modules, flat as $S_\infty / \mathfrak{a}^N$-modules. Moreover, there are canonical isomorphisms 
        \[ M_{1, N+1, \bullet} / (\mathfrak{a}^N) \cong M_{1, N, \bullet} \]
        (for $N \geq 1$) and
        \[ M_{1, 1, \bullet } / (\mathfrak{a}) \cong M_{0, \bullet}\]
        of complexes of $\mathbb{T}^S \otimes_\cO A \otimes_\cO S_\infty$-modules. 
    \end{enumerate}
\end{proposition}
\begin{proof}
    (1) The fact that $M'_\bullet \otimes_{A^W} A \to M_\bullet$ is an isomorphism is immediate from the fact that $A$ is a finite free $A^W$-module. The flatness of these complexes over $\cO$ follows from Lemma \ref{lem_inverse_limit_of_flat_is_flat}. It remains to check that the morphism $\gamma : C_\bullet\otimes_{\mathbb{T}^S} \mathbb{T}^S_{\mathfrak{m}} \to M'_\bullet$ arising by passage to the inverse limit is a quasi-isomorphism. This is more subtle. We consider the exact triangle in $\mathbf{D}(\cO)$
    \[ C_\bullet\otimes_{\mathbb{T}^S} \mathbb{T}^S_{\mathfrak{m}} \to M'_\bullet \to \mathrm{Cone}(\gamma) \to C_\bullet\otimes_{\mathbb{T}^S}\mathbb{T}^S_{\mathfrak{m}}[1]. \]
    We must show that $\mathrm{Cone}(\gamma)$ is acyclic. We claim that it is enough to show that $C_\bullet\otimes_{\mathbb{T}^S} \mathbb{T}^S_{\mathfrak{m}}$ and $M'_\bullet$ are both derived $(\varpi)$-adically complete objects of $\mathbf{D}(\cO)$, in the sense of \cite[\href{https://stacks.math.columbia.edu/tag/091N}{Section 091N}]{stacks-project}. Indeed, if this is the case, then $\mathrm{Cone}(\gamma)$ is also derived complete, as derived complete objects form a triangulated subcategory of $\mathbf{D}(\cO)$. On the other hand, $C_\bullet\otimes_{\mathbb{T}^S} \mathbb{T}^S_{\mathfrak{m}}$ and $M'_\bullet$ are bounded above complexes of flat $\cO$-modules, with $C_\bullet \otimes_{\mathbb{T}^S} \mathbb{T}^S_{\mathfrak{m}}/ (\varpi) \to M'_\bullet / (\varpi) \cong M'(1)_\bullet$ a quasi-isomorphism, by Lemma \ref{lem_inverse_limit_of_flat_is_flat} and Proposition \ref{prop_patched_complexes_torsion_level}, showing that $\mathrm{Cone}(\gamma) \otimes^{\mathbf{L}}_\cO k$ is acyclic. We can then apply \cite[\href{https://stacks.math.columbia.edu/tag/0G1U}{Lemma 0G1U}]{stacks-project} to deduce that $\mathrm{Cone}(\gamma)$ is itself acyclic. 

    By  \cite[\href{https://stacks.math.columbia.edu/tag/091U}{Lemma 091U}]{stacks-project}, derived complete objects form a weak Serre subcategory of the category of $\cO$-modules, and a complex in $\mathbf{D}(\cO)$ is derived complete if and only if its homology objects are derived complete. Since $H_\ast(C_\bullet\otimes_{\mathbb{T}^S} \mathbb{T}^S_{\mathfrak{m}})$ is a finite $\cO$-module, it is $(\varpi)$-adically complete, so derived complete by \cite[\href{https://stacks.math.columbia.edu/tag/091T}{Proposition 091T}]{stacks-project}. This shows that $C_\bullet\otimes_{\mathbb{T}^S} \mathbb{T}^S_{\mathfrak{m}}$ is derived complete. On the other hand, Lemma \ref{lem_inverse_limit_of_flat_is_flat} shows that $M'_\bullet$ is a complex of $(\varpi)$-adically complete modules,  hence of derived complete modules. Since derived complete objects form a weak Serre subcategory, this shows that the homology objects of $M'_\bullet$ are derived complete, hence that $M'_\bullet$ is derived complete, as required. 

    (2) This follows immediately from Proposition \ref{prop_patched_complexes_torsion_level} and Lemma \ref{lem_inverse_limit_of_flat_is_flat}. 

    (3) Again, this follows immediately from Proposition \ref{prop_patched_complexes_torsion_level} and Lemma \ref{lem_inverse_limit_of_flat_is_flat}. 
\end{proof}
\begin{lemma} \label{lem:delta_m0_bound}
    The sequence of modules $\delta^{n!} H_*( M_0(m)_\bullet )$, depending on $m$, is bounded, in the sense of Definition \ref{def_bounded_inverse_system}. 
\end{lemma}

\begin{proof}
    The map $\delta^{n!} = \alpha \circ \beta$ factors through the finite module $H_*(M(m)_\bullet)$. Moreover, the dimension of $H_*(M(m)_\bullet)/(\varpi)$ can be bounded independently of $m$, by Proposition \ref{prop_patched_complexes_torsion_free_level}.
\end{proof}

\begin{lemma}  \label{lem:delta_m1_bound}
The sequence of modules $\delta^{n! N} H_*( M_{1,N}(m)_\bullet )$, depending on $m$, is bounded.
\end{lemma}

\begin{proof}
    Consider the filtration $F^j = \mathfrak{a}^{j}/\mathfrak{a}^N$ of $S_\infty/(\varpi^m, \mathfrak{a}^N)$. By counting monomials in the generators of $S_\infty/(\varpi^m)$, we obtain isomorphisms $F^{j}/F^{j+1} = (S_\infty/(\varpi^m, \mathfrak{a}))^{d_j}$ with $d_j = \binom{nq + j - 1}{j}$. Since $M_{1,N}(m)_\bullet$ is a complex of flat $S_\infty/(\varpi^m, \mathfrak{a}^N)$-modules, we have the filtration
    \[
        0 = F^N(m)_\bullet \subset F^{N - 1}(m)_\bullet  \subset \cdots \subset F^0(m)_\bullet = M_{1,N}(m)_\bullet,
    \]
    where $F^{j}(m)_\bullet = F^j \otimes_{S_\infty/(\varpi^m, \mathfrak{a}^N)} M_{1,N}(m)_\bullet$. For each integer $j$, the graded piece $F^j(m)_\bullet/F^{j + 1}(m)_\bullet$ is isomorphic to 
    \[ M_{1, N}(m)_\bullet / (\mathfrak{a}) \otimes_{\cO} \mathfrak{a}^j / \mathfrak{a}^{j+1}, \]
    hence to  $M_0(m)_\bullet^{\oplus d_j}$, by Proposition \ref{prop_patched_complexes_torsion_level}.
    There is a spectral sequence
    \[
        E_1^{ij}(m) = H_{j -i}( F^{j}(m)_\bullet/F^{j+1}(m)_\bullet ) \implies H_{j - i}(M_{1,N}(m)_\bullet).
    \]
    It follows from Lemma \ref{lem:delta_m0_bound} that there exists a constant $C$ such that for each $i, j, m$ there is a surjection $\mathcal{O}^{d_j C} \to \delta^{n!} E_1^{ij}(m)$.
    Since $\mathcal{O}$ is a PID and $E_{\infty}^{ij}(m)$ is a subquotient of $E_1^{ij}(m)$, there exists a surjection $\mathcal{O}^{d_j C} \to \delta^{n!} E_\infty^{ij}(m)$. The spectral sequence induces a filtration of length $N$ on $H_*(M_{1,N}(m)_\bullet)$. Therefore, Lemma \ref{lem:filtration_generator_bound} implies that there exists a surjection
    $\mathcal{O}^{\sum d_j C} \to \delta^{n! N} H_*(M_{1,N}(m)_\bullet)$.
\end{proof}
\begin{corollary} \label{cor:patching_r1_torsion}
    We have 
    \[
        \delta^{n! N} R^1 \varprojlim_m H_*(M_{1,N}(m)_\bullet) = 0.
    \]
\end{corollary}
\begin{proof}
    The endomorphism $\delta^{n! N}$ of $R^1 \varprojlim_m H_*(M_{1,N}(m)_\bullet)$ factors through the module 
    \[
    R^1 \varprojlim_m \delta^{n! N} H_*(M_{1,N}(m)_\bullet),
    \]
    which vanishes by Lemmas \ref{lem:delta_m1_bound} and \ref{lem:bounded_sequence_limit}.
\end{proof}

\para Let $f : \mathbb{T}^S \to \cO$ be the homomorphism associated to the Hecke eigenvalues of $\iota^{-1}\pi^\infty$. We define a homomorphism $g_m : \mathbb{T}^S \otimes_\cO A^W \to \mathbf{R}_{\mathcal{F}} \otimes_{\mathbf{R}} \prod_{N \geq 1} \cO / (\varpi^m) \cong \cO / (\varpi^m)$ by $g_m(x \otimes y) = (f(x \iota_N(y)))_N$, where $\iota_N : A^W \cong A_{Q_N}^{W_{Q_N}} \to \mathbb{T}^S$. Then we define $g : \mathbb{T}^S \otimes_\cO A^W \to \cO$ by the formula $g = \varprojlim_m g_m$. Let $\mathfrak{q} = (\ker g, \mathfrak{a}) \leq \mathbb{T}^S \otimes_\cO A^W \otimes_\cO S_\infty$, and $\mathfrak{q}' = \ker f$. 
\begin{proposition}\label{prop_localised_patched_complexes}\leavevmode
    \begin{enumerate}
        \item There is an isomorphism $M'_{\mathfrak{q}, \bullet} \otimes_{A^W} A \to  M_{\mathfrak{q}, \bullet}$ of complexes of $\mathbb{T}^S \otimes_\cO A$-modules. Moreover, there are quasi-isomorphisms 
        \[ C_{\mathfrak{q}', \bullet} \to M'_{\mathfrak{q}', \bullet} \to M'_{\mathfrak{q}, \bullet} \]
        of complexes of $\mathbb{T}^S$-modules.
        \item Suppose that the Taylor--Wiles sets $Q_r$ are chosen to have the following property: there is $d \geq 1$ such that, for any $r \geq 1$, we have $\ord_{\varpi} f(\delta_{Q_r}) \leq d$. 
        Then the localisation $\alpha_{\mathfrak{q}}$ is a quasi-isomorphism 
        \[ \alpha_{\mathfrak{q}} : M_{\mathfrak{q}, \bullet} \to M_{0, \mathfrak{q}, \bullet} \]
        of complexes of $\mathbb{T}^S \otimes_\cO A$-modules. 
        \item There are canonical isomorphisms
        \[ M_{1, N+1, \mathfrak{q}, \bullet} / (\mathfrak{a}^N) \cong M_{1, N, \mathfrak{q}, \bullet} \]
        (for $N \geq 1$) and
        \[ M_{1, 1, \mathfrak{q}, \bullet } / (\mathfrak{a}) \cong M_{0, \mathfrak{q},  \bullet}\]
        of complexes of $\mathbb{T}^S \otimes_\cO A\otimes_\cO S_\infty$-modules. Moreover, $M_{1, N, \mathfrak{q}}$ is a bounded above complex of flat $S_{\infty, \mathfrak{a}} / (\mathfrak{a}^N)$-modules. 
    \end{enumerate}
\end{proposition}
\begin{proof}
    (1) Taking in hand Proposition \ref{prop_patched_complexes_torsion_free_level}, we see that the non-trivial thing to check is that the natural morphism of complexes $M'_{\mathfrak{q}', \bullet} \to M'_{\mathfrak{q}, \bullet}$ is a quasi-isomorphism. Let $T' = \mathbb{T}^S(M'_\bullet)$, $T = (\mathbb{T}^S \otimes_\cO A^W)(M'_\bullet)$. These are $\cO$-subalgebras of
    \[ \End_{\mathbf{D}(\cO)}(M'_\bullet), \]
    which is itself a finite (non-commutative) $\cO$-algebra, because $C_\bullet \cong M'_\bullet$ in $\mathbf{D}(\cO)$ has cohomology finite over $\cO$. We claim that $g$ factors through $T$. 

    To see this, we note that we in fact have $T = T'$. Indeed, we have (using the finiteness of cohomology)
    \[ (\mathbb{T}^S \otimes_\cO A^W)(M'_\bullet) = \varprojlim_m (\mathbb{T}^S \otimes_\cO A^W)(M'_\bullet / (\varpi^m) ), \]
    and the image of $A^W$ in $\End_{\mathbf{D}(\cO / (\varpi^m))}( M'_\bullet / (\varpi^m) )$ is contained in $\mathbb{T}^S$, since the action of $A^W$ can be identified, via the isomorphism
    \[ \End_{\mathbf{D}(\cO / (\varpi^m))}( M'_\bullet / (\varpi^m) ) \cong \End_{\mathbf{D}(\cO / (\varpi^m))}( C_\bullet / (\varpi^m) ), \]
    with the action of some $A_{Q_r}^{W_{Q_r}}$ ($r$ depending on $m$) via the map $\iota_r$. In particular, we can equivalently define $\mathfrak{q}$ as the pre-image in $\mathbb{T}^S \otimes_\cO A^W$ of the kernel of the map $T \to \cO$ associated to the Hecke eigenvalues of $f$ (here using that these appear in $H_\ast(C_\bullet) \otimes_\cO E$, by our choice of weight $\lambda$ and level subgroup $K$). Since we have $T'_{\mathfrak{q}'} = T_{\mathfrak{q}}$, it follows that we must also have $H_\ast(M'_{\mathfrak{q}'}) = H_\ast(M'_{\mathfrak{q}})$. 

    (2) The hypothesis implies that $\ord_\varpi g(\delta) \leq d$, hence that $\delta \not\in \mathfrak{q}$ and that $\delta$ is invertible as an endomorphism of the complexes $M_{\mathfrak{q},\bullet}$, $M_{0, \mathfrak{q}, \bullet}$ (hence certainly a quasi-isomorphism). 

    (3) This follows immediately from Proposition \ref{prop_patched_complexes_torsion_free_level}, by localisation.
\end{proof}
\begin{proposition}\label{prop_computation_of_limit_cohomology}
    Let $N \geq 1$, and suppose that there is $d \geq 1$ such that for all $r \geq 1$, $\ord_{\varpi} f(\delta_{Q_r}) \leq d$. Then there are isomorphisms
    \[ H_\ast(M_{1, N, \bullet})_{\mathfrak{q}} \cong \left( \varprojlim_m H_\ast(M_{1, N, \bullet}(m)) \right)_{\mathfrak{q}} \cong \left( \varprojlim_m \delta^{n! N} H_\ast(M_{1, N, \bullet}(m)) \right)_{\mathfrak{q}} \]
    of $\mathbb{T}^S \otimes_\cO A^W \otimes_\cO S_\infty$-modules. 
\end{proposition}
\begin{proof}
    Let $H_m = H_\ast(M_{1, N, \bullet}(m))$, $H_\infty = \varprojlim_m H_m$, $H'_m = \delta^{n! N} H_m$, and $H'_\infty = \varprojlim_m H'_m$. By \cite[\href{https://stacks.math.columbia.edu/tag/0CQE}{Lemma 0CQE}]{stacks-project}, there is a short exact sequence
    \[ 0 \to R^1 \varprojlim H_m \to H_\ast(M_{1, N, \bullet}) \to H_\infty \to 0, \]
    and therefore (using Corollary \ref{cor:patching_r1_torsion}) an isomorphism 
    \[ H_\ast(M_{1, N, \bullet})_{\mathfrak{q}} \cong H_{\infty, \mathfrak{q}}. \]
    On the other hand, there are maps $i_m : H'_m \to H_m$, $j_m : H_m \to H'_m$ such that $i_m \circ j_m$ and $j_m \circ i_m$ are both given by multiplication by $\delta^{n! N}$. By passage to the limit, we obtain maps $i : H'_\infty \to H_\infty$ and $j : H_\infty \to H'_\infty$ such that $i \circ j$ and $j \circ i$ are both given by multiplication by $\delta^{n! N}$. Our hypothesis implies that $\delta^{n! N } \not\in \mathfrak{q}$, so these maps become isomorphisms after localisation at $\mathfrak{q}$. This completes the proof. 
\end{proof}

\para This concludes our construction of patched complexes. We now discuss their relation to deformation rings. Let $D$ denote the group determinant of $G_{F, S}$ over $\cO$ associated to $\rho$, and let $\overline{D} = D \text{ mod } (\varpi)$. We write $R = R_{\overline{D}, S}$ for the associated pseudodeformation ring, as in \S \ref{sec_galois_theory}. Choose $a \leq b$ such that all the Hodge--Tate weights of $\rho$ lie in the interval $[a, b]$. Then $R$ comes equipped with the ideal $I^{ss} = I^{[a, b]}_S \leq R$, with the property that $R / I^{ss}$ classifies pseudodeformations that are semistable with Hodge--Tate weights in the interval $[a, b]$. 

We write $R_r$ for the quotient of $R_{\overline{D}, S \cup Q_r}$ classifying pseudodeformations that factor through the maximal quotient of $G_{F, S \cup Q_r}$ over which each decomposition group $G_{F_v}$ ($v \in Q_r$) is abelian; and define $I^{ss}_r = I^{[a, b]}_{S \cup Q_r} R_r$. Invoking Proposition \ref{prop_existence_of_TW_primes}, we can assume that the Taylor--Wiles data $(Q_r, ( \alpha_{v, 1}, \dots, \alpha_{v, n})_{v \in Q_r})$ have been chosen so that the following condition is satisfied: there  are integers $d, g_0, q \geq 0$ such that
\begin{itemize}
    \item For any $r \geq 1$, $\ord_{\varpi} f(\delta_{Q_r}) \leq d$ and $|Q_r| = q$.
    \item Let $\mathfrak{p} = \ker(R \to \cO)$ denote the prime ideal associated to the group determinant of $\rho$, and let $\mathfrak{p}_r = \ker(R_r \to \cO)$ denote its pullback to $R_r$. Then there is a homomorphism of $\cO$-modules $\cO^{qn -  n [ F^+ : \mathbb{Q} ] + 1} \to \mathfrak{p}_r / (\mathfrak{p}_r^2, I^{ss}_r)$ with cokernel a quotient of $(\cO / (\varpi^d))^{g_0}$.
    \item For any $v \in Q_r$, we have $v^c \not\in S$. 
\end{itemize}
For each $r \geq 1$ and $v \in Q_r$, we define $n$-dimensional group determinants of $W_{F_v}$ as follows:
\begin{itemize}
    \item We write $\alpha_v$ for the restriction of the tautological group determinant over $R_r$ from $G_{F, S \cup Q_r}$ to $W_{F_v}$.
    \item We write $\gamma_v$ for the group determinant over $\mathbb{T}^S[\Delta_{Q_r}]$ associated to the representation $\oplus_{i=1}^n t_{v, i} \circ \Art_{F_v}^{-1}$. 
\end{itemize}
The following theorem summarises what we need about the existence of Galois representations associated to the cohomology of the locally symmetric spaces $X^G_K$.
\begin{theorem}\label{thm_existence_of_Galois_action}
    We can find $N_1, N_2 \geq 0$ with the following property: for any $r \geq 1$, there is an ideal 
    \[ J_r \leq (\mathbb{T}^S[\Delta_{Q_r}])(C_{Q_r, 1, \bullet} \otimes_{\mathbb{T}^S} \mathbb{T}^S_{\mathfrak{m}})\]
    such that $J_r^{N_1} = 0$, and a map $R_r \to (\mathbb{T}^S[\Delta_{Q_r}])(C_{Q_r, 1, \bullet} \otimes_{\mathbb{T}^S} \mathbb{T}^S_{\mathfrak{m}}) / J_r$ satisfying the following conditions:
    \begin{enumerate}
        \item For any finite place $v \not\in S \cup Q_r$ of $F$, the image of the characteristic polynomial of $\Frob_v$ equals $P_v(X)$, the Hecke polynomial defined in \cite[(2.2.5)]{10author}.
        \item For any $v \in Q_r$, the pushforward of $\alpha_v$ equals the pushforward of $\gamma_v$.
        \item There is an operator $\delta_\ell \in \mathcal{H}(\GL_n(F_\ell), \GL_n(\cO_{F_\ell}))$,  not depending on $r$, and with a non-zero eigenvalue on $\pi^\infty$, such that
        \[ (\delta_\ell)^{N_2} I^{ss}_r (\mathbb{T}^S[\Delta_{Q_r}])(C_{Q_r, 1, \bullet} \otimes_{\mathbb{T}^S} \mathbb{T}^S_{\mathfrak{m}}) / J_r = 0. \]
    \end{enumerate}
\end{theorem}
\begin{proof}
    It follows from \cite[Theorem 1.6]{Hev25} (applied with $S_{\text{avoid}} = S_\ell(F)$, $S_{\text{bad}} = S \cup Q_r$) that we can find an ideal $J_r$ of bounded nilpotence degree and a map $R_r \to (\mathbb{T}^S[\Delta_{Q_r}])(C_{Q_r, 1, \bullet} \otimes_{\mathbb{T}^S} \mathbb{T}^S_{\mathfrak{m}}) / J_r$ satisfying (1) and (2). It follows from \cite[Proposition 5.2.8, Lemma 5.2.7]{Aca25} that, after possibly enlarging $J_r$ while keeping the nilpotence degree bounded, we can further assume $\delta_\ell I_r^{ss} (\mathbb{T}^S[\Delta_{Q_r}])(C_{Q_r, 1, \bullet} \otimes_{\mathbb{T}^S} \mathbb{T}^S_{\mathfrak{m}}) / J_r = 0$, where $\delta_\ell$ is the operator denoted $T_n^2 p^{N_2}$ in the statement of \cite[Lemma 5.2.7]{Aca25}, and $T_n$ is as defined in \cite[\S 5.4.2]{Aca25}. It has a non-zero eigenvalue on $\pi^\infty$ by construction. This completes the proof.
\end{proof}

\para Our goal is to deduce a corresponding statement for the patched deformation ring, that we now introduce. Let us define for any $m \geq 1$
\[ R^p_m = \mathbf{R}_{\mathcal{F}} \otimes_{\mathbf{R}} \prod_{r \geq 1} R_r / \mathfrak{m}_{R_r}^m, \]
and 
\[ R^p = \varprojlim R^p_m. \]
Let $I^p_m \leq R^p_m$ denote the image of the ideal $\prod_{r \geq 1} I^{ss}_r$, and let $I^p = \varprojlim_m I^p_m$.
\begin{lemma}\label{lem_properties_of_patched_deformation_ring}\leavevmode
\begin{enumerate}
    \item $R^p$ is a complete Noetherian local $\cO$-algebra.
    \item The image of $I^p$ under the natural homomorphism $R^p \to R$ is equal to $I^{ss}$.
    \item Let $\mathfrak{p}^p \leq R^p$ denote the pre-image of $\mathfrak{p} \leq R$. Then there is a surjective $E$-algebra homomorphism $E\llbracket X_1, \dots, X_{nq -  n [ F^+ : \mathbb{Q} ] + 1} \rrbracket \to \widehat{R}^p_{\mathfrak{p}^p} / (I^p)$.
\end{enumerate}
\end{lemma}
\begin{proof}
        (1) The rings $R_r$ may be presented as quotients of $\cO \llbracket X_1, \dots, X_{g_0} \rrbracket$ (for a fixed $g_0$, independent of $r$). Therefore each ring $R^p_m$ is an Artinian local $\cO$-algebra, and $R^p$ is a complete Noetherian local $\cO$-algebra.

    (2) It suffices to show that $I^{p} R / (\mathfrak{m}_{R_r}^m) = I^{ss} R / (\mathfrak{m}_{R_r}^m)$ for each $m \geq 1$. The image of $\prod_{r \geq 1} I^{ss}_r R_r / \mathfrak{m}_{R_r}^m$ in $\prod_{r \geq 1} R / \mathfrak{m}_{R}^m$ is equal to $\prod_{r \geq 1} I^{ss} R / \mathfrak{m}_{R}^m$, so this is true.

    (3) We recall that for each $r \geq 1$, there are elements \[
    x_{r, 1}, \dots, x_{r, nq -  n [ F^+ : \mathbb{Q} ] + 1} \in \mathfrak{p}_r\] such that $\mathfrak{p}_r/(\mathfrak{p}_r^2, I^{ss}_r, x_{r,1}, \dots, x_{r,nq - n[F^+:\mathbb{Q}]+1})$ is a quotient of $(\cO / (\varpi^d))^{g_0}$. Let $x_i$ denote the image of $(x_{r, i})$ in $R^p$. It suffices to show that the $\mathcal{O}$-module $\mathfrak{p}^p / (I^p, x_1, \dots, x_{nq -  n [ F^+ : \mathbb{Q} ] + 1})$ is torsion; or even that \[\varpi^d \mathfrak{p}^p R^p_m \subset (I^p, x_1, \dots, x_{nq -  n [ F^+ : \mathbb{Q} ] + 1}) R^p_m\] for every $m \geq 1$. 

    The ideal $\mathfrak{p}^p R^p_m$ is the image of the ideal $\prod_{r \geq 1} \mathfrak{p}_r R^p_m \leq \prod_{r \geq 1} R_r / \mathfrak{m}_{R_r}^m$. To prove the desired containment, it is therefore enough to show that $(I^p, x_1, \dots, x_{nq -  n [ F^+ : \mathbb{Q} ] + 1}) R^p_m$ equals the image of the ideal
    \[ \prod_{r \geq 1} (I^{ss}_r, x_{r, 1}, \dots, x_{r, nq -  n [ F^+ : \mathbb{Q} ] + 1}) \leq \prod_{r \geq 1} R_r / \mathfrak{m}_{R_r}^m. \]
    Since $I^p R^p_m$ is equal to the image of $\prod I^{ss}_r$ by definition, it suffices to show that $(x_1, \dots, x_{nq -  n [ F^+ : \mathbb{Q} ] + 1}) R^p_m$ is equal to the image of \[
    \prod_r (x_{r, 1}, \dots, x_{r, nq -  n [ F^+ : \mathbb{Q} ] + 1}) R_r / \mathfrak{m}_{R_r}^m;\] equivalently, that if $I = (x_1, \dots, x_{nq -  n [ F^+ : \mathbb{Q} ] + 1}) R^p_m$, then we have the equality
    \[ I (\prod_{r \geq 1} R_r / \mathfrak{m}_{R_r}^m) = \prod_{r \geq 1} (I R_r / \mathfrak{m}_{R_r}^m). \]
    However, this follows from the fact that $I$ is finitely generated. This completes the proof. 
\end{proof}

\para We need to explain what happens at Taylor--Wiles places. Fix, for each $r \geq 1$, an ordering of $Q_r$, and for each $v \in Q_r$, a surjection $\mathbb{Z}_p \to k(v)^\times(p)$. These choices determine a surjection $S_\infty \to \cO [ \Delta_{Q_r} ]$. We thus obtain, for each $r \geq 1$, a tuple $\alpha_{r, 1}, \dots, \alpha_{r, q}$ of $n$-dimensional group determinants of $\mathbb{Z} \times \mathbb{Z}_p$ over $R_r$. Similarly, we have characters $\gamma_{i}^{(j)} : \mathbb{Z} \times \mathbb{Z}_p \to (S_\infty \otimes_\cO A)^\times$ ($j = 1, \dots, n$) with associated $n$-dimensional group determinants $\gamma_i = \oplus_{j=1}^n \gamma_i^{(j)}$ ($i = 1, \dots, q$) over $S_\infty \otimes_\cO A$, whose pushforward to any $\mathbb{T}^S[\Delta_{Q_r}]$ coincides with $\gamma_{v_1}, \dots, \gamma_{v_q}$.

We can use the $\alpha_{r, i}$ to define group determinants $\alpha_{\infty, i}$ over $R^p$. Indeed, the product $\prod_{r \geq 1} \alpha_{r, i}$ gives an $n$-dimensional group determinant of $\mathbb{Z} \times \mathbb{Z}_p$ over $\prod_{r \geq 1} R_r / \mathfrak{m}_{R_r}^m$ for every $m \geq 1$, so passage to inverse limit and then localization gives the desired $\alpha_{\infty, i}$.
\begin{lemma}\label{lem_separation_of_characters}
Fix $1 \leq i \leq q$.
    \begin{enumerate}
        \item The pushforward of $\alpha_{\infty, i}$ to $R$ is inflated from the `unramified quotient' $\mathbb{Z} \times \mathbb{Z}_p \to \mathbb{Z}$. The image of the characteristic polynomial $\alpha_{\infty, i}(X - (1, 0))$ in $(R /\mathfrak{p})[X] = \cO[X]$ has distinct roots, contained in $\cO$, which come with a canonical ordering $x_i^{(1)}, \dots, x_i^{(n)}$.
        \item There are unique characters $\alpha_{\infty, i}^{(j)} : \mathbb{Z} \times \mathbb{Z}_p \to (\widehat{R}^p_{\mathfrak{p}^p})^\times$ ($i = 1, \dots, q$, $j = 1, \dots, n$) such that $\alpha_{\infty, i}$ is the group determinant of the representation $\oplus_{j=1}^n \alpha_{\infty, i}^{(j)}$ and $\alpha_{\infty, i}^{(j)}((1,0)) \equiv x_i^{(j)} \text{ mod }\mathfrak{p}^p$.
        \item Let $\widehat{S}_{\infty, \mathfrak{a}} \to \widehat{R}^p_{\mathfrak{p}^p}$ be the $E$-algebra homomorphism associated to the characters $\alpha_{\infty, i}^{(j)}$. Then  $\widehat{R}^p_{\mathfrak{p}^p} / (\mathfrak{a}) = \widehat{R}_{\mathfrak{p}}$.
        \item For any $s \geq 1$, the map $\mathbb{T}^S \otimes_\cO S_\infty \otimes_\cO E \to R^p_{\mathfrak{p}^p} / (\mathfrak{p}^p)^s$ is surjective. 
    \end{enumerate}
\end{lemma}
\begin{proof}
The first two parts may be proved in exactly the same way as \cite[Lemma 4.23, 4.24]{New23}. For the third, we need to work harder. 

Let $F_1 / F$ denote the (finite) extension cut out by $\overline{\rho}$, and let $M / F_1$ be the maximal pro-$p$ extension contained in $F_S$. Then $\Gal(M/ F)$ is topologically finitely generated, and (\cite[Lemma 3.8]{Che14}) any pseudodeformation of $\overline{D}$, unramified outside of $S$, factors through $\Gal(M / F)$. Similarly, if $r \geq 1$, let $M_r / F_1$ be the maximal pro-$p$ extension contained in $F_{S \cup Q_r}^{Q_r-ab}$. Then $\Gal(M_r/ F)$ is topologically finitely generated, and we can find $m_0 \geq 0$ such that for any $r \geq 1$, it can be topologically generated by $m_0$ elements (cf. \cite[Lemma 3.28]{Tho15}).

Let $\mathcal{G}_{m_0}$ denote the free profinite group on $m_0$ elements, and choose for each $r \geq 1$ a continuous, surjective homomorphism $\mathcal{G}_{m_0} \to \Gal(M_r / F)$. We extend this to a surjection $\psi_r : \mathcal{G}_{2q + m_0} \to \Gal(M_r / F)$ by sending the first $q$ generators $f_1, \dots, f_q$ to the images of the elements $(0, 1) \in \mathbb{Z} \times \mathbb{Z}_p$ (in other words, to our marked generators for the tame inertia at the $i$th place of $Q_r$), and the second $q$ generators $f_{q+1}, \dots, f_{2q}$ to the images of the elements $(1, 0)$ (the marked Frobenius lifts). 

We now apply some results from \cite{Che14}. Let $A_r \in \CNL_\cO$ denote the object representing the functor of continuous pseudodeformations of $\psi_r^\ast \det \overline{\rho}$. 
Let $\mathcal{I}_r \leq A_r$ denote the ideal such that a group determinant factors through the quotient of $\mathcal{G}_{2q+m_0}$ by $f_1, \dots, f_q$ if and only if it factors through $A_r / \mathcal{I}_r$. By construction, there is a surjective homomorphism $A_r \to R_r$, and $R_r / \mathcal{I}_r R_r = R$.

Let $A^p_m = \mathbf{R}_{\mathcal{F}} \otimes_{\mathbf{R}} \prod_{r \geq 1} A_r / \mathfrak{m}_{A_r}^m$, $A^p = \varprojlim A^p_m$, $\mathcal{I}^p = \varprojlim_m \mathbf{R}_{\mathcal{F}} \otimes_{\mathbf{R}} \prod_{r \geq 1} \mathcal{I}_r A_r / \mathfrak{m}_{A_r}^m$. Then $A^p \in \CNL_\cO$, by \cite[Corollary 2.39]{Che14}, $\mathcal{I}^p \leq A^p$ is a finitely generated ideal, and $R^p$ is an $A^p$-algebra. Moreover, $A^p$ carries a continuous group determinant $D^p$ of $\mathcal{G}_{2q + m_0}$, and we can naturally identify $R^p / \mathcal{I}^p R^p = R$. Localising and completing at $\mathfrak{p}^p$, we see that the map $\widehat{R}^p_{\mathfrak{p}^p} \to \widehat{R}_{\mathfrak{p}}$ factors through an isomorphism $\widehat{R}^p_{\mathfrak{p}^p} / (\mathcal{I}^p) \cong \widehat{R}_{\mathfrak{p}}$. To prove the 3rd point of the lemma, it thus suffices to show that $\mathcal{I}^p \widehat{R}^p_{\mathfrak{p}^p} = \mathfrak{a} \widehat{R}^p_{\mathfrak{p}^p}$. 

To show this, note that $D^p \text{ mod }\mathfrak{p}$ is absolutely irreducible, so $D^p$ is the group determinant of a unique (up to equivalence) representation $\rho^p : \mathcal{G}_{2q+m_0} \to \GL_n(\widehat{R}^p_{\mathfrak{p}^p})$. The quotient $\widehat{R}^p_{\mathfrak{p}^p} / \mathcal{I}^p \widehat{R}^p_{\mathfrak{p}^p}$ may be interpreted as the maximal quotient over which $\rho^p$ factors through the quotient of $\mathcal{G}_{2q+m_0}$ by the closed subgroup normally generated by the elements $f_1, \dots, f_q$; in other words, the maximal quotient over which we have $\rho^p(f_i) = 1$ for each $i = 1, \dots, q$. On the other hand, by construction, the group determinant $\alpha_{\infty, i}$ ($i = 1, \dots, q$) of $\mathbb{Z} \times \mathbb{Z}_p$ may be realised as the restriction of $\rho^p$ to the subgroup generated by $f_i$ and $f_{i+q}$. Since this restriction is isomorphic to a direct sum of the characters $\alpha_{\infty, i}^{(1)}, \dots, \alpha_{\infty, i}^{(n)}$, the result follows. 

For the fourth part of the lemma, it suffices to show that  $\mathbb{T}^S \otimes_\cO E \to \widehat{R}^p_{\mathfrak{p}^p} / (\mathfrak{a}, (\mathfrak{p}^p)^s)$ is surjective. This is true because $\mathbb{T}^S \to R / \mathfrak{p}^s$ is surjective, by the Chebotarev density theorem (cf. \cite[Proposition 3.26]{Tho15}).
\end{proof}
We have defined a homomorphism $g : \mathbb{T}^S \otimes_\cO A^W \to \cO$. Let $\widetilde{g} : \mathbb{T}^S \otimes_\cO A \otimes_\cO S_\infty \to \cO$ be the extension of $g$ defined by sending $t_i^{(j)}$ to $x_i^{(j)}$ ($i = 1, \dots, q$, $j = 1, \dots, n$) and $\mathfrak{a}$ to $0$. Let $\widetilde{\mathfrak{q}} = \ker \widetilde{g}$. Then the localised complex $M_{1, N, \widetilde{\mathfrak{q}}, \bullet}$ is defined, and is in fact quasi-isomorphic to a direct summand of $M_{1, N, \mathfrak{q}, \bullet}$ in $\mathbf{D}(\widehat{S}_{\infty, \mathfrak{a}} / \mathfrak{a}^N)$. 
\begin{theorem}\label{thm_existence_of_pached_Galois_action}
     We can find an integer $N_3 \geq 0$ with the following property: for any $N \geq 1$, there is an ideal
    \[ K_N \leq (\mathbb{T}^S \otimes_\cO S_\infty \otimes_\cO A)(M_{1, N, \bullet})_{\widetilde{\mathfrak{q}}} \]
    such that $K_N^{N_3} = 0$, and a homomorphism of $S_\infty$-algebras: 
    \[ \widehat{R}^p_{\mathfrak{p}^p} / I^p \to (\mathbb{T}^S \otimes_\cO S_\infty \otimes_\cO A)(M_{1, N, \bullet})_{\widetilde{\mathfrak{q}}} / K_N \]
    satisfying the following conditions:
    \begin{enumerate}
        \item For each $i = 1, \dots, q$, the pushforward of $\gamma_{ i}$ is equal to the pushforward of $\alpha_{\infty, i}$. 
        \item The diagram
        \[ \xymatrix{ \widehat{R}^p_{\mathfrak{p}^p} / I^p \ar[r] \ar[d] & (\mathbb{T}^S \otimes_\cO S_\infty \otimes_\cO A)(M_{1, N, \bullet})_{\widetilde{\mathfrak{q}}} / K_N \ar[d] \\
        \widehat{R}_{\mathfrak{p}} / I^{ss} \ar[r] & (\mathbb{T}^S \otimes_\cO S_\infty \otimes_\cO A)(M_{\bullet})_{\widetilde{\mathfrak{q}}} / (K_N) }
              \]
    commutes.
    \end{enumerate}
\end{theorem}
\begin{proof}
    Let $H_m = H_\ast(M_{1, N, \bullet}(m))$, $H'_m = \delta^{n! N} H_m$. We first show that each $H'_m$ is an $R^p$-module in a natural way. Note that
    \[ H_m = \mathbf{R}_{\mathcal{F}} \otimes_{\mathbf{R}} \prod_{r \geq m+N} H_\ast(C_{Q_r, 1, \bullet} / (\mathfrak{a}^N) \otimes_{\mathbb{T}^S} \mathbb{T}^S_{\mathfrak{m}} / (\varpi^m)), \]
    and
    \[ H'_m = \mathbf{R}_{\mathcal{F}} \otimes_{\mathbf{R}} \prod_{r \geq m+N} \delta^{n! N} H_\ast(C_{Q_r, 1, \bullet} / (\mathfrak{a}^N)\otimes_{\mathbb{T}^S} \mathbb{T}^S_{\mathfrak{m}} / (\varpi^m)) = \mathbf{R}_\mathcal{F} \otimes_{\mathbf{R}} \prod_{r \geq m+N} H'_{r, m}, \]
    if we define
    \[ H'_{r, m} = \delta^{n! N} H_\ast(C_{Q_r, 1, \bullet}  / (\mathfrak{a}^N)\otimes_{\mathbb{T}^S} \mathbb{T}^S_{\mathfrak{m}} / (\varpi^m)). \]
    Arguing as in the proof of Lemma \ref{lem:delta_m1_bound}, we see that the groups $H'_{r, m}$ are finite $\cO$-modules of length bounded independently of $r$. Define
    \[B_{r}(m) = (\mathbb{T}^S \otimes_\cO S_\infty \otimes_{\cO} A)(H'_{r, m}). \]
    By \cite[Lemma 2.2.6]{AC23}, there is a natural identification
    \[ \End_\cO(H'_m) = \mathbf{R}_{\mathcal{F}} \otimes_{\mathbf{R}} \prod_{r \geq m+N} \End_\cO(H'_{r, m}). \]
    Let $B(m) \leq \End_\cO(H'_m)$ correspond to $\mathbf{R}_{\mathcal{F}} \otimes_{\mathbf{R}} \prod_{r \geq m+N} B_r(m)$. Then $B(m)$ is a $\mathbb{T}^S \otimes_\cO S_\infty \otimes_\cO A$-algebra (although, as far as we can see, the map $\mathbb{T}^S \otimes_\cO S_\infty \otimes_\cO A \to B(m)$ may not be surjective).
    
    Theorem \ref{thm_existence_of_Galois_action} implies that there is a morphism $R_r \to B_{r}(m) / J_{r}(m)$ satisfying the following conditions:
    \begin{itemize}
        \item $J_{r}(m)$ is a nilpotent ideal satisfying $J_{r}(m)^{N_1} = 0$.
        \item For any finite place $v\not\in S \cup Q_r$ of $F$, the image of the characteristic polynomial of $\Frob_v$ equals $P_v(X)$.
        \item For any $v \in Q_r$, the pushforward of $\alpha_v$ equals the pushforward of $\gamma_v$. 
        \item We have $\delta_\ell^{N_2} I_r^{ss} B_{r}(m) / J_{r}(m) = 0$.
    \end{itemize}
    The ring $B_{r}(m)$ is an Artinian local $\cO$-algebra, with maximal ideal $\mathfrak{m}_{B_{r}(m)}$ nilpotent, of exponent bounded independently of $r$ (although depending on $N$ and $m$). In particular, there exists $n(m)$ such that the map $R_r \to B_{r}(m) / J_{r}(m)$ factors through $R_r / \mathfrak{m}_{R_r}^{n(m)}$. By passage to product and localisation, we obtain a map
    \[ R^p_{n(m)} \to \left( \mathbf{R}_{\mathcal{F}} \otimes_{\mathbf{R}} \prod_{r \geq m+N} B_{r}(m) \right) / J(m), \] 
    where $J(m)$ is the image of the ideal $\prod_{r \geq N} J_{r}(m)$, hence a map 
    \[ R^p \to B(m) / J(m). \] 
    We would like to pass to the limit. Let $H'_\infty = \varprojlim_m H'_m$, and let $H''_m = \operatorname{im}(H'_\infty \to H'_m)$, $B'(m)$ the quotient of $B(m)$ that acts faithfully on $H''_m$. Then the rings $B'(m)$ form an inverse system, and $B'(\infty) = \varprojlim_m B'(m)$ acts faithfully on $H'_\infty$. Let $J'(m) = J(m) B'(m)$, and let $f_m : R^p \to B'(m) / J'(m)$ be the induced map. 
    By construction, the map $R^p \to B'(m) / J'(m+1) B'(m)$ induced by $f_{m+1}$ is equal to $f_m$ once we quotient by $J'(m) + J'(m+1)$. Therefore, if we set $J''(m) = \ker(B'(m) \to \prod_{i=1}^m B'(i) / J'(i))$, then there is a unique map $f'_m : R^p \to B'(m) / J''(m)$ compatible with the $f_i$, $i = 1, \dots, m$, in this sense, and the rings $B'(m) / J''(m)$ form an inverse system. Let $J''(\infty) = \ker (B'(\infty) \to \prod_{m=1}^\infty B'(m) / J''(m))$, and let $f'_\infty : R^p \to B'(\infty) / J''(\infty)$ denote the induced map. By construction, $B'(\infty)$ is a $\mathbb{T}^S \otimes_\cO S_\infty \otimes_\cO A$-algebra, and the following conditions are satisfied:
    \begin{itemize}
        \item $J''(\infty)^{N_1} = 0$.
        \item For each $i = 1, \dots, q$, the pushforward of $\alpha_{\infty, i}$ equals the pushforward of $\gamma_i$.
        \item We have $\delta_\ell^{N_2} I^p B'(\infty) / J''(\infty) = 0$.
    \end{itemize}
    Localising at the prime $\widetilde{\mathfrak{q}} \leq \mathbb{T}^S \otimes_\cO S_\infty \otimes_\cO A$, we obtain a map
    \[ \widehat{R}^p_{\mathfrak{p}^p} \to B'(\infty)_{\widetilde{\mathfrak{q}}} / J''(\infty), \,\,\, B'(\infty)_{\widetilde{\mathfrak{q}}} \leq \End_{\widehat{S}_{\infty, \mathfrak{a}}/(\mathfrak{a}^N)}(H'_{\infty, \widetilde{\mathfrak{q}}}) \]
    (note the map from the localisation of $R^p$ extends to the completion because the target is an Artinian $E$-algebra). By Proposition \ref{prop_computation_of_limit_cohomology}, there is an isomorphism
    \begin{equation}\label{eqn_cohomology} H'_{\infty, \widetilde{\mathfrak{q}}} \cong H_\ast(M_{1, N, \bullet})_{\widetilde{\mathfrak{q}}}. 
    \end{equation}
    Lemma \ref{lem_separation_of_characters}(4) implies that the image of the map $\widehat{R}^p_{\mathfrak{p}^p} \to B'(\infty)_{\widetilde{\mathfrak{q}}} / (J''(\infty))$ 
    is contained in the image of the map
    \[ (\mathbb{T}^S \otimes_\cO S_\infty \otimes_\cO A)_{\widetilde{\mathfrak{q}}} \to B'(\infty)_{\widetilde{\mathfrak{q}}} \to B'(\infty)_{\widetilde{\mathfrak{q}}}/(J''(\infty)). \]
    This image is itself a quotient of $(\mathbb{T}^S \otimes_\cO S_\infty \otimes_\cO A)(M_{1, N, \bullet})_{\widetilde{\mathfrak{q}}}$ by a nilpotent ideal $K_N$ of exponent $N_3 = N_1 (\dim_\mathbb{R} X^G + 1)$, by \eqref{eqn_cohomology} and \cite[Lemma 2.2.3]{10author}. We finally obtain a diagram
      \[ \xymatrix{ \widehat{R}^p_{\mathfrak{p}^p}  \ar[r] \ar[d] & (\mathbb{T}^S \otimes_\cO S_\infty \otimes_\cO A)(M_{1, N, \bullet})_{\widetilde{\mathfrak{q}}} / K_N \ar[d] \\
        \widehat{R}_{\mathfrak{p}}\ar[r] & (\mathbb{T}^S \otimes_\cO S_\infty \otimes_\cO A)(M_{\bullet})_{\widetilde{\mathfrak{q}}} / (K_N) }
              \]
    To complete the proof, it remains to justify the following points:
    \begin{itemize}
        \item The diagram commutes. 
        \item The respective kernels of the horizontal arrows contain $I^p$, $I^{ss}$. 
        \item For each $i = 1, \dots, q$, and for each $j = 1, \dots, n$, the pushforward of $\alpha_{\infty, i}^{(j)}$ equals the pushforward of $\gamma_i^{(j)}$. 
    \end{itemize}
    The first point holds by construction. The second holds because we have $\delta_\ell^{N_2} I^p B'(\infty)_{\widetilde{\mathfrak{q}}} / (J''(\infty)) = 0$, by construction, and because $\delta_\ell \not\in \widetilde{\mathfrak{q}}$ (as $\mathfrak{q}$ captures the Hecke eigenvalues of a cuspidal automorphic representation of $\GL_n(\mathbb{A}_F)$, which is generic at the unramified places, by \cite[Corollary 5.10]{Sha74} -- this uses the definition of $\delta_\ell$ in \cite[\S 2]{Aca25}). The third point holds because the pushforwards of $\alpha_{\infty, i}$ and $\gamma_i$ are equal, by construction; these pseudocharacters are residually multiplicity-free; and they agree modulo $\widetilde{\mathfrak{q}}$, so we can apply \cite[Proposition 1.5.1]{Bel09}. This completes the proof. 
\end{proof}

\begin{definition} We define $M_{1, \infty, \bullet} = \varprojlim_N M_{1, N, \widetilde{\mathfrak{q}}, \bullet}$.
\end{definition}

\begin{corollary}\label{cor_descent_from_infinite_level}\leavevmode
    \begin{enumerate}
        \item $M_{1, \infty, \bullet}$ is a complex of $\mathbb{T}^S \otimes_\cO \widehat{S}_{\infty, \mathfrak{a}} \otimes_\cO A$-modules, perfect as a complex of $\widehat{S}_{\infty, \mathfrak{a}}$-modules; and there is a quasi-isomorphism $M_{1, \infty, \bullet} / (\mathfrak{a}) \to M_{\widetilde{\mathfrak{q}}, \bullet}$ of complexes of $\mathbb{T}^S \otimes_\cO A$-modules. 
        \item There is a nilpotent ideal $K_\infty \leq (\mathbb{T}^S \otimes_\cO \widehat{S}_{\infty, \mathfrak{a}} \otimes_\cO A)(M_{1, \infty, \bullet})$ and a morphism 
    \[ \widehat{R}^p_{\mathfrak{p}^p} / I^p \to (\mathbb{T}^S \otimes_\cO \widehat{S}_{\infty, \mathfrak{a}} \otimes_\cO A)(M_{1, \infty, \bullet}) / K_\infty \]
    of $\widehat{S}_{\infty, \mathfrak{a}}$-algebras such that the diagram
    \[ \xymatrix{ \widehat{R}^p_{\mathfrak{p}^p} / I^p \ar[r] \ar[d] & (\mathbb{T}^S \otimes_\cO \widehat{S}_{\infty, \mathfrak{a}} \otimes_\cO A)(M_{1, \infty, \bullet})/ K_\infty \ar[d] \\
        \widehat{R}_{\mathfrak{p}} / I^{ss} \ar[r] & (\mathbb{T}^S \otimes_\cO \widehat{S}_{\infty, \mathfrak{a}} \otimes_\cO A)(M_{\bullet})_{\widetilde{\mathfrak{q}}} / (K_\infty) }
              \]
    commutes.
    \end{enumerate}
\end{corollary}
\begin{proof}
    The first point follows from the properties of $M_{1, N, \widetilde{\mathfrak{q}}, \bullet}$, and Lemma \ref{lem_inverse_limit_of_flat_is_flat}. 
    For the second, note that we have
    \[ (\mathbb{T}^S \otimes_\cO \widehat{S}_{\infty, \mathfrak{a}}\otimes_\cO A)(M_{1, \infty, \bullet} ) = \varprojlim_N (\mathbb{T}^S \otimes_\cO \widehat{S}_{\infty, \mathfrak{a}} \otimes_\cO A)(M_{1, N, \widetilde{\mathfrak{q}}, \bullet}) \]
    (using e.g.\ \cite[Lemma 2.13]{Kha17}), so this follows from Theorem \ref{thm_existence_of_pached_Galois_action} by passage to the limit. 
\end{proof}
We are now ready to complete the proof of Theorem \ref{thm_vanishing_of_Selmer_special_case}.
\begin{proof}[Proof of Theorem \ref{thm_vanishing_of_Selmer_special_case}]
We apply Lemma \ref{lem_patching_lemma} with $S = \widehat{S}_{\infty, \mathfrak{a}}$, $T_\infty = (\mathbb{T}^S \otimes_\cO \widehat{S}_{\infty, \mathfrak{a}} \otimes_\cO A)(M_{1, \infty, \bullet})$, $I_\infty = K_\infty$, and $M_{\infty, \bullet} = M_{1, \infty, \bullet}$. We need to check hypothesis (2) of the lemma, namely that $H_\ast(M_{1, \infty, \bullet} / (\mathfrak{a}))$ is a semisimple $T_\infty$-module. By Corollary \ref{cor_descent_from_infinite_level} and Proposition \ref{prop_complexes_compute_classical_cohomology}, we can identify
\[ H_\ast(M_{1, \infty, \bullet} / (\mathfrak{a})) \cong H_\ast(M_{\widetilde{\mathfrak{q}}, \bullet}) \cong H_\ast( X_K^G, \mathcal{V}_\lambda^\vee )_{\mathfrak{q}'}, \]
so this follows from strong multiplicity one and the computation, using Franke's theorem, of the cohomology of $X_K^G$ in terms of automorphic forms: following the strategy of \cite[Theorem 2.4.10]{10author}, we see that only $\pi$ contributes to the displayed group. 

By Lemma \ref{lem_properties_of_patched_deformation_ring}, $\widehat{R}^p_{\mathfrak{p}^p} / (I^p)$ is a quotient of $E \llbracket x_1, \dots, x_{nq - n [ F^+ : \mathbb{Q} ] + 1} \rrbracket$, while $\widehat{S}_{\infty, \mathfrak{a}}$ is isomorphic to a power series ring in $nq$ variables, and $H_\ast( M_{1, \infty, \bullet} / (\mathfrak{a}))$ is non-zero in a range of degrees of length $n [ F^+ : \mathbb{Q} ] - 1$, by the computation of the cuspidal cohomology of $\GL_n$ (cf. the proof of \cite[Theorem 2.4.10]{10author} for a very similar argument, using the results of \cite{Fra98a} and \cite{Clo90}). The last thing to check is that the map 
\[ \widehat{R}^p_{\mathfrak{p}^p} / (I^p) \to (\mathbb{T}^S \otimes_\cO \widehat{S}_{\infty, \mathfrak{a}} \otimes_\cO A)(M_{1, \infty, \bullet}) / K_\infty \]
is surjective. However, this is true because the right-hand side is generated, as $\widehat{S}_{\infty, \mathfrak{a}}$-algebra, by $\mathbb{T}^S$ and the values of the characters $\gamma_i^{(j)} = \alpha_{\infty, i}^{(j)}$, which are in the image of $\widehat{R}^p_{\mathfrak{p}^p}$.

We can therefore apply the conclusion of the lemma, which states that the map
\[ E\llbracket x_1, \dots, x_{nq - n [F^+ : \mathbb{Q}] + 1} \rrbracket \to \widehat{R}^p_{\mathfrak{p}^p} / (I^p)\to (\mathbb{T}^S \otimes_\cO \widehat{S}_{\infty, \mathfrak{a}} \otimes_\cO A)(M_{1, \infty, \bullet}) / K_\infty  \]
is an isomorphism, and that $\mathfrak{a} \widehat{R}^p_{\mathfrak{p}^p}/ (I^p) = \mathfrak{p}^p \widehat{R}^p_{\mathfrak{p}^p}/ (I^p)$. Applying Lemma \ref{lem_properties_of_patched_deformation_ring}(2) and Lemma \ref{lem_separation_of_characters}(3), we conclude that $\widehat{R}_{\mathfrak{p}} / (I^{ss}) = E$. The Zariski tangent space to $\widehat{R}_{\mathfrak{p}} / (I^{ss})$ may be identified with $H^1_f(F, \ad \rho \otimes_\cO E)$ (same proof as \cite[Proposition 2.17]{New23}), so this completes the proof.
\end{proof}

\section{Deduction of main theorems}\label{sec_general_results}

\begin{theorem}\label{thm_general_vanishing_result}
    Let $F$ be a CM field, let $\pi$ be a cuspidal, regular algebraic automorphic representation of $\GL_n(\mathbb{A}_F)$, let $p$ be a prime, and let $\iota : \overline{\mathbb{Q}}_p \to \mathbb{C}$ be an isomorphism. Suppose that the following conditions are satisfied:
    \begin{enumerate}
        \item $r_{\pi, \iota}(G_{F(\zeta_{p^\infty})})$ is enormous.
        \item Writing $L / F(\zeta_{p^\infty})$ for the extension cut out by $\ad r_{\pi,\iota}$, we have \[
        H^1(L / F, \ad r_{\pi,\iota}(1)) = 0.\] 
        \item For each finite place $v$ of $F$ such that $v | p$ or $\pi_v$ is ramified, $r_{\pi, \iota}|_{G_{F_v}}$ is generic.
    \end{enumerate}
    Then $H^1_f(F, \ad r_{\pi, \iota}) = 0 $.
\end{theorem}
\begin{proof}
    We are free to replace $F$ by a finite CM extension, linearly disjoint from $L$, as this only increases the size of the relevant Selmer group. (Recall that we are working with characteristic 0 coefficients, so restriction is injective.) By \cite[Lemma 5.1]{New23}, and cyclic base change \cite{Art89}, we can assume that $\pi$ is everywhere unipotently ramified, and that $F$ contains an imaginary quadratic field $F_0$ such that any rational prime which is either equal to $p$, or which lies below a ramified place of $\pi$, splits in $F_0$. (We need to check that if $M / F$ is a soluble CM extension such that $\mathrm{BC}_{M / F}(\pi)$ is cuspidal, and $w$ is a finite place of $M$ lying above the place $v$ of $F$, then $r_{\pi, \iota}|_{G_{M_w}} = r_{\mathrm{BC}_{M / F}(\pi), \iota}|_{G_{M_w}}$ is generic. If $r_{\pi, \iota}|_{G_{F_v}}$ is generic, then 
    there is an isomorphism
    \[ \mathrm{WD}(r_{\pi, \iota}|_{G_{F_v}})^{F-ss} \cong \rec_{F_v}^T(\iota^{-1} \pi_v). \]
    The local representation $\pi_v$ is essentially unitarizable and generic. The classification of unitary representations of $\GL_n(F_v)$ \cite{Tad86} shows that $\rec_{F_v}^T(\iota^{-1} \pi_v)|_{W_{M_w}}$ is generic, and we have
    \[ \rec_{F_v}^T(\iota^{-1} \pi_v)|_{W_{M_w}} \cong \mathrm{WD}(r_{\pi, \iota}|_{G_{M_w}})^{F-ss}. \]
    Therefore $r_{\pi, \iota}|_{G_{M_w}}$ is generic, as required.) Making a further extension of $F$ of sufficiently large degree, in which all the $p$-adic places split, we can assume that hypothesis (6) of Theorem \ref{thm_vanishing_of_Selmer_special_case} is satisfied. All the other hypotheses are satisfied, by construction, so the result now follows.
\end{proof}
We deduce some corollaries that may be easier to apply in practice.
\begin{corollary}\label{cor_selmer_vanishing_under_purity}
     Let $F$ be a CM field, let $\pi$ be a cuspidal, regular algebraic automorphic representation of $\GL_n(\mathbb{A}_F)$, let $p$ be a prime, and let $\iota : \overline{\mathbb{Q}}_p \to \mathbb{C}$ be an isomorphism. Suppose that the following conditions are satisfied:
     \begin{enumerate}
         \item $r_{\pi, \iota}(G_{F(\zeta_{p^\infty})})$ is enormous.
         \item For each finite place $v$ of $F$, $\mathrm{WD}(r_{\pi, \iota}|_{G_{F_v}})$ is pure, in the sense of \cite[\S 1]{Tay07}.
     \end{enumerate}
     Then $H^1_f(F, \ad r_{\pi, \iota}) = 0 $.
\end{corollary}
\begin{proof}
    We need to check that the hypotheses of Theorem \ref{thm_general_vanishing_result} are satisfied. The vanishing of $H^1(L / F, \ad r_{\pi, \iota}(1))$ follows from purity and \cite[Lemma 6.2]{kisin2003}. The genericity also follows from purity. 
\end{proof}
\begin{corollary}\label{cor_selmer_vanishing_under_polarizability}
    Let $F$ be a CM field, let $\pi$ be a cuspidal, polarizable, regular algebraic automorphic representation of $\GL_n(\mathbb{A}_F)$, let $p$ be a prime, and let $\iota : \overline{\mathbb{Q}}_p \to \mathbb{C}$ be an isomorphism. Suppose that $r_{\pi, \iota}(G_{F(\zeta_{p^\infty})})$ is enormous. Then $H^1_f(F, \ad r_{\pi, \iota}) = 0$.
\end{corollary}
We note that this result is strictly stronger than \cite[Theorem A]{New23}, because we consider the whole adjoint representation of $\GL_n$ (as opposed to just the conjugate self-dual part). 
\begin{proof}
    This follows from Corollary \ref{cor_selmer_vanishing_under_purity} because in this case, $r_{\pi, \iota}$ is known to be pure (cf. \cite[Theorem 2.1.1]{Bar14} and the other papers cited there).
\end{proof}
\begin{corollary}\label{cor_selmer_vanishing_for_elliptic_curves}
    Let $F$ be a CM field, let $E / F$ be an elliptic curve with $\End(E_{\overline{F}}) = \mathbb{Z}$, and let $p$ be a prime. Then for any $n \geq 1$, we have $H^1_f(F, \ad \Sym^n V_p E) = 0$.
\end{corollary}
\begin{proof}
    By \cite[Corollary 7.1.12]{10author}, $\Sym^n V_p E$ is potentially automorphic. Since we are working in characteristic 0, restriction is injective on Selmer groups, so we can assume that $\Sym^n V_p E$ is automorphic. It follows from Serre's open image theorem and \cite[Lemma 2.28]{New23} that the enormous image condition of Corollary \ref{cor_selmer_vanishing_under_purity} is satisfied. The Galois representations arising from elliptic curves are likewise known to be pure, so the result follows. 
\end{proof}

\bibliographystyle{alpha}
\bibliography{refs}
\end{document}